\documentclass{amsart}

\makeatletter
\@namedef{subjclassname@2020}{%
  \textup{2020} Mathematics Subject Classification}
\makeatother
\newtheorem{theorem}{Theorem}[section]

\theoremstyle{definition}

\newtheorem{corollary}[theorem]{Corollary}

\newtheorem{lemma}[theorem]{Lemma}
\theoremstyle{remark}

\numberwithin{equation}{section}

\begin{document}

\title[Gradient estimates for a weighted parabolic equation]{Gradient estimates for a weighted parabolic equation under geometric flow }

\author{Shahroud Azami}
\address{Department of  Pure Mathematics, Faculty of Science,
Imam Khomeini International University,
Qazvin, Iran. \\
              Tel.: +98-28-33901321\\
              Fax: +98-28-33780083\\
             }

\email{azami@sci.ikiu.ac.ir}
%
%

\subjclass[2020]{53C21, 53E20, 35K55, 35B45}



\keywords{Gradient estimate, Harnack inequality, Parabolic equation, Geometric flow}
\begin{abstract}
Let $(M^{n},g,e^{-\phi}dv)$  be a weighted  Riemannian manifold  evolving by geometric flow $\frac{\partial g}{\partial t}=2h(t),\,\,\,\frac{\partial \phi}{\partial t}=\Delta \phi$.  In this paper, we obtain  a series of space-time  gradient estimates   for positive solutions of  a parabolic partial equation
$$(\Delta_{\phi}-\partial_{t})u(x,t)=q(x,t)u^{a+1}(x,t)+p(x,t)A(u(x,t))),\,\,\,\,(x,t)\in M\times[0,T]$$

on a weighted Riemannian manifold under geometric flow. By integrating the gradient estimates, we find the corresponding Harnack inequalities.
\end{abstract}
\maketitle

\section{Introduction}
 An $n$-dimensional smooth weighted Riemannain  manifold (or smooth metric measure space ) $(M^{n}, g, e^{-\phi}dv)$  is an  $n$-dimensional smooth  Riemannain  manifold $(M^{n},g)$ endowed with  a weighted volume element $e^{-\phi}dv$ such that $\phi\in C^{2}(M)$ and $dv$ is the volume element of $g$ on $M$. The weighted Laplacian (or Witten-Laplace operator) $\Delta_{\phi}$ is a symmetric diffusion operator and is defined by $\Delta-\nabla \phi.\nabla$.
In present paper, we will prove  Li-Yau type, local elliptic gradient estimate and another gradient estimates for positive  solution of  the parabolic equation
\begin{equation}\label{e1}
(\Delta_{\phi}-\partial_{t})u(x,t)=q(x,t)u^{a+1}(x,t)+p(x,t) A(u(x,t)),\,\,\,\,(x,t)\in M\times[0,T],
\end{equation}
 on  a weighted Riemannian manifold $(M^{n},g,e^{-\phi}dv)$ evolving by the geometric flow system
\begin{equation}\label{e2}
\frac{\partial }{\partial t}g(x,t)=2h(x,t),\,\,\,\,\frac{\partial }{\partial t}\phi=\Delta \phi
\end{equation}
where $(x,t)\in M\times[0,T]$, $p(x,t), q(x,t)$ are  functions on $M\times[0,T]$ of $C^{2}$ in $x$-variables and  $C^{1}$ in $t$-variable, $A(u)$ is a function  of $C^{2}$ in $u$, $a$ is a positive constant, and $h(x,t)$ is a symmetric $(0,2)$-tensor field on $(M,g(t),e^{-\phi}dv)$.  \\
Some examples of geometric flows are the Ricci flow \cite{RH} when $h=-Ric$ where $Ric$  is the Ricci tensor, Yamabe flow \cite{BC} when $h=-\frac{1}{2}Rg$ where $R$ is the scalar curvature, Ricci-Bourguignon flow \cite{GCL} when $h=-Ric+\rho R g$ where $\rho$ is constant, and  the extended Ricci flow \cite{BLI} when $h=-Ric+\alpha\nabla\phi\otimes\nabla\phi$ where  $\alpha(t)$   is a nonincreasing function  and $\phi$ is a smooth scalar function.\\

The equation (\ref{e1}) is the so called reaction-diffusion equation which can be found  in many mathematical models in physics, chemistry, and biology \cite{FR, JS}, where $qu^{a+1}+pA(u)$ and $\Delta_{\phi}$
are the reaction and the diffusion terms, respectively. For instance, when $\phi$ is  a constant function, $a=0$, and $A(u)=u\log u$, the nonlinear elliptic equation corresponding  to (\ref{e1}) is related to the gradient Ricci soliton. When $\phi$ is a constant function and  $A(u)=u^{b}$, the equation corresponding  to (\ref{e1}) is related to the Yamabe type equation.    When $\phi$ is a constant function and $A(u)=0$  then $(\Delta-\partial_{t})u=qu^{a+1}$ which is a simple ecological model for population dynamics.\\

Gradient and Harnack estimates are powerful tools and important techniques in heat kernel analysis, entropy theory, differential geometry, in particulary, in studying solution  of parabolic equations from geometry which it developed by P. Li and S.-T. Yau \cite{LY}. In fact, they proved the well-known Li-Yau estimate on positive solutions to the heat equation with potential on Riemannian manifold with a fixed  Riemannian metric and Ricci curvature bounded from below. Then, they derived Harnack inequalities by integrating the global gradient estimate along the  a space-time path which provides  a comparison between  heat  at two different points in space and at different times. After than, this method  plays powerful role  in study of heat equation, in particular, geometric flows. For instance,  R. S. Hamilton \cite{RSH1} proved  a harnack estimate for Ricci flow on Riemannain manifolds with weakly positive curvature operator which is used  in solving the Poincar{\'e} conjecture \cite{HDC, GP}.\\

In 1993, R. S. Hamilton \cite{RSH2} obtained  an elliptic type gradient estimate for positive solutions of the heat equations on compact manifolds which was known  as the Hamilton type gradient estimate. Then, for complete noncompact manifold, P. Souplet and Q. S. Zhang \cite{SZ} established an elliptic type gradient estimate for bounded solutions of the heat equation by adding  a logarithmic correction term. This is called the Souplet-Zhang type gradient estimate. Li-Yau  type, Hamilton type, and Souplet-Zhang type gradient estimates have been obtained for other nonlinear parabolic equations on  manifolds, for instance see \cite{LC, QCH, NTD, YLI, LMA, ERN, QR, JYWU} and the references therein.
On the other hand, many authors used similar techniques  to prove gradient estimates and Harnack  inequalities for positive solutions of parabolic equations under the geometric flow, see for example \cite{MBA, XCH, HGU, MIS, YLZ, JS, GZH}. \\

In 2014, X. Zhu and  Y. Li \cite{ZYL} derived Li-Yau estimates for  a parabolic equation of the type $(\Delta -q -\partial_{t})u=au(\log u)^{\alpha}$ in $M\times(0,\infty)$ with a fixed metric where $a,\alpha$ are constants  and $q\in C^{2}(M\times(0,\infty))$.  In \cite{QCH}, Q. Chen and G. Zhao studied  the equation $(\Delta-q-\partial_{t})u=A(u)$ with a convection terms on  a complete manifold with a fixed metric where $A(u)$  is a function  of $C^{2}$ in $u$. Then, in \cite{GZH}, G. Zhao obtained Li-Yau type and Hamilton type gradient estimates  of  equation $(\Delta-q-\partial_{t})u=A(u)$ on Riemannian manifold evolving by the geometric flow. In \cite{JYW},  J. Y. Wu gave  a local Li-Yau Type gradient estimate for the positive solutions to a nonlinear parabolic equation $\partial_{t} u=\Delta_{\phi} u-au\ln u- qu$ in $M\times [0,T]$, where  $a$ is  a real constant and $q\in C^{2}(M\times(0,\infty))$. Also, J. Y. Wu \cite{JYW} proved  local Hamilton type and Souplet-Zhang type gradient estimates for positive solutions to the equation $\partial_{t} u=\Delta_{\phi} u+au\ln u$ with $a\in\mathbb{R}$ on a smooth metric measure space $(M^{n},g,e^{-\phi}dv)$ with Bakry-{\'E}mery Ricci tensor is bounded from below. In 2019, F. Yang and L. Zhang \cite{FYZ} proved Li-Yau type, Hamilton type, Souplet-Zhang type, and  the fourth type gradient estimates for positive solutions of  a nonlinear parabolic equation $(\Delta_{\phi}-\partial_{t})u=pu+qu^{a+1}$ on smooth metric measure space with a fixed metric. \\

In this paper, we establish  some gradient estimates for bounded positive solution of (\ref{e1}) under the geometric flow (\ref{e2}), which are  richer than \cite{FYZ, GZH}.\\
In following we recall some basic definitions on an $n$-dimensional weighted Riemannian manifold $(M,g, e^{-\phi}dv)$.
The weighted Bochner formula for any smooth function $f$ is follow as
 \begin{equation}\label{e4}
\frac{1}{2}\Delta_{\phi}|\nabla f|^{2}=|{\rm Hess} f|^{2}+\langle \nabla\Delta_{\phi}f,\nabla f\rangle+Ric_{\phi}(\nabla f,\nabla f),
\end{equation}
where
 \begin{equation*}
Ric_{\phi}:=Ric+{\rm Hess}\phi
\end{equation*}
and it is called a Bakry-\'{E}mery tensor  (see \cite{DB}). For any integer $m>n$, an $(m-n)$-Bakry-\'{E}mery tensor (see \cite{DB1}) is defined by
 \begin{equation*}
Ric_{\phi}^{m-n}:=Ric+{\rm Hess}\phi-\frac{\nabla\phi\otimes\nabla\phi}{m-n}.
\end{equation*}
If $f=\ln u$ then $\Delta_{\phi}u=u(\Delta_{\phi}f+|\nabla f|)^{2})$   and $\partial_{t} u=u\partial_{t}f$. Therefore, by (\ref{e1}) the function $f$ satisfies
 \begin{equation}\label{e3}
(\Delta_{\phi}-\partial_{t})f=-|\nabla f|^{2}+qe^{af}+p\hat A(f),
\end{equation}
where $\hat A(f)=\frac{A(u)}{u}$.
Throughout the paper, we assume $u$ be  a positive  smooth solution to the general parabolic equation (\ref{e1}). We denote by $n$ the dimension  of the manifold $M$, and by $d(x,y,t)$ the geodesic distance between $x,y\in M$ under $g(t)$. In addition, for any fixed $x_{0}\in M$, $R>0$ we define the compact set
 \begin{equation*}
Q_{2R,T}:=\{(x,t):d(x,x_{0},t)\leq 2R, 0\leq t\leq T\}\subset M^{n}\times (-\infty,+\infty).
\end{equation*}
Let $f=\ln u$ and $\hat{A}(f)=\frac{A(u)}{u}$. Then
 \begin{equation*}
\hat{A}_{f}=A'(u)-\frac{A(u)}{u},\,\,\,\,\,\hat{A}_{ff}=uA''(u)-A'(u)+\frac{A(u)}{u}.
\end{equation*}
Moreover, for $u>0$ we define several nonnegative real constants as follows:
\begin{eqnarray*}
&&\lambda_{1}:=\mathop{\sup}\limits_{Q_{2R,T}}|\hat{A}|,\qquad\qquad\quad\quad\,\,\Lambda_{1}:=\mathop{\sup}\limits_{M\times [0,T]}|\hat{A}|,\\
&&\lambda_{2}:=\mathop{\sup}\limits_{Q_{2R,T}}|\hat{A}_{f}|,\qquad\qquad\quad\,\,\Lambda_{2}:=\mathop{\sup}\limits_{M\times [0,T]}|\hat{A}_{f}|,\\
&&\lambda_{3}:=\mathop{\sup}\limits_{Q_{2R,T}}|\hat{A}_{ff}|,\qquad\qquad\quad\,\,\Lambda_{3}:=\mathop{\sup}\limits_{M\times [0,T]}|\hat{A}_{ff}|,
\end{eqnarray*}
\begin{eqnarray*}
&&\gamma_{1}:=\mathop{\sup}\limits_{Q_{2R,T}}|p|,\qquad\qquad\quad\,\,\Gamma_{1}:=\mathop{\sup}\limits_{M\times [0,T]}|p|,\\
&&\gamma_{2}:=\mathop{\sup}\limits_{Q_{2R,T}}|\nabla p|,\qquad\qquad\,\,\Gamma_{2}:=\mathop{\sup}\limits_{M\times [0,T]}|\nabla p|,\\
&&\gamma_{3}:=\mathop{\sup}\limits_{Q_{2R,T}}|\Delta_{\phi}p|,\qquad\qquad\Gamma_{3}:=\mathop{\sup}\limits_{M\times [0,T]}|\Delta_{\phi}p|.
\end{eqnarray*}
and
\begin{eqnarray*}
&&\sigma_{1}:=\mathop{\sup}\limits_{Q_{2R,T}}|q|,\qquad\qquad\quad\,\,\Sigma_{1}:=\mathop{\sup}\limits_{M\times [0,T]}|q|,\\
&&\sigma_{2}:=\mathop{\sup}\limits_{Q_{2R,T}}|\nabla q|,\qquad\qquad\,\,\Sigma_{2}:=\mathop{\sup}\limits_{M\times [0,T]}|\nabla q|,\\
&&\sigma_{3}:=\mathop{\sup}\limits_{Q_{2R,T}}|\Delta_{\phi}q|,\qquad\qquad\Sigma_{3}:=\mathop{\sup}\limits_{M\times [0,T]}|\Delta_{\phi}q|.
\end{eqnarray*}
Also,
\begin{eqnarray*}
&&\theta_{1}:=\mathop{\sup}\limits_{Q_{2R,T}}|\nabla \phi|,\qquad\qquad\,\,\Theta_{1}:=\mathop{\sup}\limits_{M\times [0,T]}|\nabla \phi|,\\
&&\theta_{2}:=\mathop{\sup}\limits_{Q_{2R,T}}|\nabla\Delta\phi|,\qquad\qquad\Theta_{2}:=\mathop{\sup}\limits_{M\times [0,T]}|\nabla\Delta\phi|.
\end{eqnarray*}
The rest of this paper is organized as follows.\\
In Section 2, we give a Li-Yau type gradient estimate for positive solution  of (\ref{e1}) under the geometric flow (\ref{e2}). We firstly prove  a local and a global Li-Yau type gradient estimate on complete noncompact  weighted Riemannian manifold without boundary (see Theorem \ref{t1} and Corollary \ref{c1}) and as an immediate consequence  of the Corollary  \ref{c1}, by integrating the global gradient estimate in space-time we establish the corresponding Harnack inequality (see Corollary \ref{c2}). Then we consider that weighted manifold $M$ is closed   and we obtain global Li-Yau type gradient estimate and its corresponding Harnack inequality for  positive solution  of (\ref{e1}) on $M$ under geometric flow (\ref{e2}) (see Theorem \ref{t2} and Corollary \ref{c3}). In Sections 3 and 4, we prove    local and global Hamilton  type  and  Souplet-Zhang type gradient estimates on complete noncompact  weighted Riemannian manifold without boundary  for  positive solution  of (\ref{e1}) on $M$ under geometric flow (\ref{e2}), respectively (see Theorem \ref{t3},  Corollary \ref{c4}, Theorem \ref{t4},  Corollary \ref{c5} ). Finally, in Section 5, similar as  as  in  \cite{YLZ, FYZ}, we obtain  a local and a global another type gradient estimate and the corresponding Harnack inequality to global estimate  for positive solution of (\ref{e1}) under the geometric flow (\ref{e2})  on on complete noncompact  weighted Riemannian manifold without boundary (see Throrem \ref{t5}, Corollary \ref{c6} and Corollary \ref{c7}).

\section{Li-Yau type gradient estimates}
Firstly, we give a local space-time Li-Yau gradient estimate  for (\ref{e1})-(\ref{e2}) with conditions of  $Ric_{\phi}^{m-n}$ is lower bounded.
\begin{theorem}\label{t1}
Let $(M,g(0),e^{-\phi_{0}}dv )$ be a complete weighted Riemannian manifold, and let $g(t), \phi(t)$ evolve by (\ref{e2}) for $t\in [0,T]$. Given $x_{0}$ and $R>0$, let $u$ be a positive solution to (\ref{e1}) in $Q_{2R,T}$ such that $u^{a}\leq k$ for some positive constant $k$. Suppose that there exist constants $k_{1}, k_{2}, k_{3}, k_{4}$ such that
\begin{equation*}
Ric_{\phi}^{m-n}\geq -(m-1)k_{1}g,\qquad -k_{2}g\leq h\leq k_{3}g,\qquad |\nabla h|\leq k_{4},
\end{equation*}
on $Q_{2R,T}$. Then for any $\alpha>1$ and $\delta\in(0,1)$,  there exist positive constants $c_{0},c_{1},$ and $c_{2}$ such that
  \begin{equation}\label{10}
\frac{|\nabla u|^{2}}{u^{2}}-\alpha qu^{a}-\alpha p \frac{  A(u)}{u}-\alpha \frac{u_{t}}{u}\leq
 \frac{m\alpha^{2}}{2t(1-\epsilon\alpha)}+K
\end{equation}
on $Q_{2R,T}$,  where $\epsilon \in (0,\frac{1}{\alpha})$,
\begin{eqnarray*}
K&:=&
\frac{m\alpha^{2}}{2(1-\epsilon\alpha)}\Big(\frac{m\alpha^{2}c_{1}}{4(1-\epsilon\alpha)(\alpha-1)R^{2}}+\gamma_{1}\lambda_{2}+a\sigma_{1}k \\\nonumber&&+\frac{c_{0}}{R}(m-1)(\sqrt{k_{1}}+\frac{2}{R})+\frac{3c_{1}}{R^{2}}+c_{2}k_{2} \Big)
+\Big(\frac{m\alpha^{2}}{2(1-\epsilon\alpha)}E \Big)^{\frac{1}{2}},
\end{eqnarray*}
\begin{eqnarray*}\nonumber
E&:=&
\frac{3}{4}
\Big(\frac{2m\alpha^{2}}{(1-\epsilon\alpha)\delta (\alpha-1)^{2}}\Big)^{\frac{1}{3}} [\alpha(a+1)-1]^{\frac{4}{3}}k^{\frac{4}{3}}\sigma_{2}^{\frac{4}{3}}\\\nonumber&&
+ \frac{3}{4}
\Big(\frac{m\alpha^{2}}{2(1-\epsilon\alpha)(1-\delta) (\alpha-1)^{2}}\Big)^{\frac{1}{3}} \alpha^{\frac{4}{3}}\theta_{2}^{\frac{4}{3}}+\alpha k\sigma_{3}
\\\nonumber&&
+\frac{3}{4}t_{0}^{2}\Big(\frac{m\alpha^{2}}{2(1-\epsilon\alpha)(1-\delta) (\alpha-1)^{2}}\Big)^{\frac{1}{3}} \big( (\alpha-1)\lambda_{1} +\alpha\lambda_{2}\big)^{\frac{4}{3}}\gamma_{2}^{\frac{4}{3}}+\alpha t_{0}^{2}\lambda_{1} \gamma_{3}
\\&&
+ \frac{m\alpha^{2}}{4(1-\epsilon\alpha)(1-\delta)(\alpha-1)^{2}}C^{2}
\\\nonumber&&
+\frac{\alpha  n}{2\epsilon}(k_{2}+k_{3})^{2}
+\frac{9}{8}n\alpha^{2}k_{4}+2\alpha  k_{2}\epsilon \theta_{1}^{2},
\end{eqnarray*}
and
\begin{eqnarray*}
C&:=&(\alpha-1)( \gamma_{1}\lambda_{2}+2a\sigma_{1} k +2k_{3})+\alpha \gamma_{1} \lambda_{3}+2(1-\epsilon\alpha)(m-1)k_{1}\\&&+\frac{\alpha k_{2}}{2\epsilon}+\alpha a^{2}\sigma_{1} k+2k_{4}+a\sigma_{2}k.
\end{eqnarray*}
\end{theorem}
For prove our results, we need the following lemmas. From \cite{BC1} we have
\begin{lemma}\label{l1}
Let the metric evolves by (\ref{e2}). Then for any smooth function $f$, we have
 \begin{equation*}
\frac{\partial }{\partial t}|\nabla f|^{2}=-2h(\nabla f,\nabla f)+2\langle \nabla f,\nabla f_{t}\rangle
\end{equation*}
and
 \begin{eqnarray*}
(\Delta_{\phi} f)_{t}&=&\Delta_{\phi}f_{t}-2\langle h, {\rm Hess} f\rangle-2\langle {\rm div }h-\frac{1}{2}\nabla({\rm tr}_{g} h), \nabla f\rangle\\&&+2h(\nabla \phi,\nabla f)-\langle \nabla f,\nabla \phi_{t}\rangle\\
&=&\Delta_{\phi}f_{t}-2\langle h, {\rm Hess} f\rangle-2\langle {\rm div }h-\frac{1}{2}\nabla({\rm tr}_{g} h), \nabla f\rangle\\&&+2h(\nabla \phi,\nabla f)-\langle \nabla f,\nabla \Delta\phi\rangle
\end{eqnarray*}
where ${\rm div }h$ is the divergence of $h$.
\end{lemma}
\begin{lemma}\label{l2}
Let  $(M^{n},g,e^{-\phi}dv)$  be a weighted  Riemannian manifold, $g(t)$ evolves by (\ref{e2}) for $t\in[0,T]$ satisfies the hypotheses of Theorem \ref{t1}. If $f=\ln u$ and $F:=t\Big(|\nabla f|^{2}-\alpha qe^{af}-\alpha p { \hat A}(f)-\alpha f_{t} \Big)$, then  for any $\epsilon \in (0,\frac{1}{\alpha})$  we have
 \begin{eqnarray}\nonumber
(\Delta_{\phi}-\partial_{t})F&\geq&
\frac{2(1-\epsilon\alpha)t}{m}\Big( |\nabla f|^{2}-f_{t}-qe^{af}-p\hat{A}\Big)^{2}-2 \langle \nabla F,\nabla f\rangle-\frac{F}{t}
\\\nonumber&&
+\alpha t p\hat{A}_{f}\Big( |\nabla f|^{2}-f_{t}-qe^{af}-p\hat{A}\Big)-2t[\alpha(a+1)-1]e^{af}\langle \nabla q,\nabla f\rangle
\\\nonumber&&
-\alpha t a q e^{af}\Big( |\nabla f|^{2}-f_{t}-qe^{af}-p\hat{A}\Big)-\alpha t\langle \nabla f,\nabla \phi_{t}\rangle
-\alpha t e^{af}\Delta_{\phi}q\\\nonumber&&
-2t\big( (\alpha-1)\hat{A}+\alpha \hat{A}_{f}\big)\langle \nabla f,\nabla p\rangle-\alpha t \Delta_{\phi}p
\\\label{007}&&
-t\Big(2(\alpha-1)p\hat{A}_{f} +\alpha  p\hat{A}_{ff}+2(1-\epsilon\alpha)(m-1)k_{1}+\frac{\alpha k_{2}}{2\epsilon}\\\nonumber&&+2(\alpha-1)qae^{af}+\alpha a^{2}qe^{af}+2(\alpha-1)k_{3}\Big)|\nabla f|^{2}
\\\nonumber&&
-\frac{\alpha t n}{2\epsilon}(k_{2}+k_{3})^{2}
-3\alpha t\sqrt{n}k_{4}|\nabla f|-2\alpha tk_{2}\epsilon |\nabla \phi|^{2}.
\end{eqnarray}
\end{lemma}
\begin{proof}
From (\ref{e3}) we have
 \begin{equation}\label{1}
\Delta_{\phi}f=-\frac{F}{t}-(\alpha-1)\big(p\hat{A}+qe^{af}+f_{t}\big).
\end{equation}
By the weighted  Bochner formula, (\ref{e3}) and Lemma \ref{l1}, we calculate
\begin{eqnarray}\nonumber
\Delta_{\phi}F&=&2t|{\rm Hess} f|^{2}+2tRic_{\phi}(\nabla f,\nabla f)+2t \langle \nabla\Delta_{\phi}f,\nabla f\rangle-\alpha t \Delta_{\phi}f_{t}\\\nonumber&&-\alpha t e^{af}\Delta_{\phi}q
-\alpha t p\hat{A}_{f}\Delta_{\phi} f-\alpha t p\hat{A}_{ff}|\nabla f|^{2}
-\alpha t a^{2}qe^{af}|\nabla f|^{2}\\\nonumber&&
-\alpha t \hat{A}\Delta_{\phi}p
-2\alpha t \hat{A}_{f}\langle \nabla p,\nabla f\rangle-\alpha t a q e^{af}\Delta_{\phi}f
\\\nonumber&&
-2t\alpha a e^{af} \langle \nabla q,\nabla f\rangle\\\label{2}&=&
2t|{\rm Hess} f|^{2}+2tRic_{\phi}(\nabla f,\nabla f)+2t \langle \nabla\Delta_{\phi}f,\nabla f\rangle-\alpha t (\Delta_{\phi} f)_{t}\\\nonumber&&-2\alpha t\langle h, {\rm Hess} f\rangle
-2\alpha t\langle {\rm div }h-\frac{1}{2}\nabla({\rm tr}_{g} h), \nabla f\rangle+2\alpha th(\nabla \phi,\nabla f)\\\nonumber&&-\alpha t\langle \nabla f,\nabla \phi_{t}\rangle
-\alpha t e^{af}\Delta_{\phi}q-\alpha t p\hat{A}_{f}\Delta_{\phi} f-\alpha t p\hat{A}_{ff}|\nabla f|^{2}\\\nonumber&&
-\alpha t \hat{A}\Delta_{\phi}p
-2\alpha t \hat{A}_{f}\langle \nabla p,\nabla f\rangle-\alpha t a^{2}qe^{af}|\nabla f|^{2}\\\nonumber&&
-\alpha t a q e^{af}\Delta_{\phi}f-2t\alpha a e^{af} \langle \nabla q,\nabla f\rangle
\end{eqnarray}
By (\ref{1}) we have
 \begin{equation}\label{3}
\nabla\Delta_{\phi}f=-\frac{\nabla F}{t}-(\alpha-1)\big(p\hat{A}_{f}\nabla f+\hat{A}\nabla p+e^{af}\nabla q+qae^{af}\nabla f+\nabla f_{t}\big)
\end{equation}
 and
 \begin{equation}\label{4}
(\Delta_{\phi}f)_{t}=\frac{F}{t^{2}}-\frac{F_{t}}{t}-(\alpha-1)\big(p\hat{A}_{f}f_{t}+p_{t}\hat{A}+q_{t}e^{af}+aqf_{t}e^{af}+f_{tt}\big).
\end{equation}
Plugging (\ref{3}) and (\ref{4}) into (\ref{2}), we obtain
\begin{eqnarray}\nonumber
\Delta_{\phi}F&=&
2t|{\rm Hess} f|^{2}+2tRic_{\phi}(\nabla f,\nabla f)-2 \langle \nabla F,\nabla f\rangle -2t(\alpha-1)p\hat{A}_{f}|\nabla f|^{2}\\\nonumber&&-2t(\alpha-1)\hat{A}\langle \nabla p,\nabla f\rangle
-2t(\alpha-1)e^{af}\langle \nabla q,\nabla f\rangle
-2t(\alpha-1)qae^{af}|\nabla f|^{2}\\\nonumber&&-2t(\alpha-1)\langle \nabla f_{t},\nabla f\rangle
-\alpha \frac{F}{t}+\alpha F_{t}-\alpha t\hat{A}\Delta_{\phi}p-2\alpha t\hat{A}_{f}\langle \nabla p,\nabla f\rangle\\\nonumber&&+\alpha t (\alpha-1)\big(p\hat{A}_{f}f_{t}+p_{t}\hat{A}+q_{t}e^{af}+aqf_{t}e^{af}+f_{tt}\big)
\\\label{5}&&-2\alpha t\langle h, {\rm Hess} f\rangle
-2\alpha t\langle {\rm div }h-\frac{1}{2}\nabla({\rm tr}_{g} h), \nabla f\rangle+2\alpha th(\nabla \phi,\nabla f)\\\nonumber&&-\alpha t\langle \nabla f,\nabla \phi_{t}\rangle
-\alpha t e^{af}\Delta_{\phi}q-\alpha t p\hat{A}_{f}\Delta_{\phi} f-\alpha t p\hat{A}_{ff}|\nabla f|^{2}\\\nonumber&&
-\alpha t a^{2}qe^{af}|\nabla f|^{2}-\alpha t a q e^{af}\Delta_{\phi}f-2t\alpha a e^{af} \langle \nabla q,\nabla f\rangle.
\end{eqnarray}
On the other hand, by Lemma (\ref{l1}) we derive
\begin{eqnarray}\nonumber
F_{t}&=&|\nabla f|^{2}-\alpha qe^{af}-\alpha {p \hat A}(f)-\alpha f_{t} \\\label{6}&&
-2th(\nabla f, \nabla f)+2t\langle \nabla f, \nabla f_{t}\rangle-\alpha t\big(p\hat{A}_{f}f_{t}+p_{t}\hat{A}+q_{t}e^{af}+aqf_{t}e^{af}+f_{tt}\big).
\end{eqnarray}
Applying (\ref{5}) and (\ref{6}), we get
 \begin{eqnarray}\nonumber
(\Delta_{\phi}-\partial_{t})F&=&
2t|{\rm Hess} f|^{2}+2tRic_{\phi}(\nabla f,\nabla f)-2 \langle \nabla F,\nabla f\rangle-\frac{F}{t}\\\nonumber&& -2\alpha t\langle h, {\rm Hess} f\rangle
-2\alpha t\langle {\rm div }h-\frac{1}{2}\nabla({\rm tr}_{g} h), \nabla f\rangle+2\alpha th(\nabla \phi,\nabla f)\\\nonumber&&
-2t(\alpha-1)p\hat{A}_{f}|\nabla f|^{2}-2t(\alpha-1)\hat{A}\langle \nabla p,\nabla f\rangle
\\\label{7}&&-2t(\alpha-1)e^{af}\langle \nabla q,\nabla f\rangle-2t(\alpha-1)qae^{af}|\nabla f|^{2}-\alpha t\langle \nabla f,\nabla \phi_{t}\rangle
\\\nonumber&&-\alpha t e^{af}\Delta_{\phi}q+\alpha t p\hat{A}_{f}\Big( |\nabla f|^{2}-f_{t}-qe^{af}-p\hat{A}\Big)-\alpha t p\hat{A}_{ff}|\nabla f|^{2}\\\nonumber&&-\alpha t\hat{A}\Delta_{\phi}p-2\alpha t\hat{A}_{f}\langle \nabla p,\nabla f\rangle
-\alpha t a^{2}qe^{af}|\nabla f|^{2}\\\nonumber&&-\alpha t a q e^{af}\Big( |\nabla f|^{2}-f_{t}-qe^{af}-p\hat{A}\Big)\\\nonumber&&-2t\alpha a e^{af} \langle \nabla q,\nabla f\rangle
-2(\alpha-1)th(\nabla f, \nabla f).
\end{eqnarray}
We can write  the boundeness condition on $h_{ij}$ as $-(k_{2}+k_{3})g_{ij}\leq h_{ij}\leq (k_{2}+k_{3})g_{ij}$ so that
\begin{equation}\label{07}
|h|^{2}\leq n(k_{2}+k_{3})^{2},
\end{equation}
since $h_{ij}$ is a symmetric tensor. The Young's inequality, for any $\epsilon\in(0,\frac{1}{\alpha})$ implies that
  \begin{equation}\label{8}
\langle h, {\rm Hess} f\rangle\leq \epsilon|{\rm Hess} f|^{2}+\frac{1}{4\epsilon}|h|^{2}\leq \epsilon|{\rm Hess}f|^{2}+\frac{n}{4\epsilon}(k_{2}+k_{3})^{2}.
\end{equation}
Also, we have
  \begin{equation}\label{9}
|{\rm div} h-\frac{1}{2}\nabla({\rm tr}_{g}h)|=|g^{ij}\nabla_{i}h_{jl}-\frac{1}{2}g^{ij}\nabla_{l}h_{ij}|\leq \frac{3}{2}|g||\nabla h|\leq \frac{3}{2}\sqrt{n}k_{4}.
\end{equation}
Notice also that for any $m>n$ we derive
\begin{eqnarray*}
0&\leq& \Big( \sqrt{\frac{m-n}{mn}}\Delta f +\sqrt{\frac{n}{m(m-n)}}\langle \nabla f, \nabla \phi \rangle\Big)^{2}\\&=&
(\frac{1}{n}-\frac{1}{m})(\Delta f)^{2}+\frac{2}{m}\Delta f \langle \nabla f, \nabla \phi \rangle+(\frac{1}{m-n}-\frac{1}{m})\langle \nabla f, \nabla \phi \rangle ^{2}\\&\leq&
|{\rm Hess} f|^{2}-\frac{1}{m}\Big((\Delta f)^{2}-2\Delta f \langle \nabla f, \nabla \phi \rangle+\langle \nabla f, \nabla \phi \rangle^{2} \Big)+\frac{1}{m-n}\langle \nabla f, \nabla \phi \rangle^{2}\\&=&
|{\rm Hess} f|^{2}-\frac{(\Delta_{\phi} f)^{2}}{m}+\frac{1}{m-n}\langle \nabla f, \nabla \phi \rangle^{2}.
\end{eqnarray*}
Therefore
  \begin{equation}\label{a10}
|{\rm Hess} f|^{2}\geq \frac{(\Delta_{\phi} f)^{2}}{m}-\frac{1}{m-n}\langle \nabla f, \nabla \phi \rangle^{2}.
\end{equation}
Substituting (\ref{8}), (\ref{9}),  and (\ref{a10}) into (\ref{7}) we conclude
 \begin{eqnarray}\nonumber
(\Delta_{\phi}-\partial_{t})F&\geq&
\frac{2(1-\epsilon\alpha)t}{m}\Big( |\nabla f|^{2}-f_{t}-qe^{af}-p\hat{A}\Big)^{2}+2\alpha t h(\nabla \phi, \nabla f)
\\\nonumber&&+2t(1-\epsilon\alpha)Ric_{\phi}^{m-n}(\nabla f ,\nabla f) -2 \langle \nabla F,\nabla f\rangle-\frac{F}{t}
\\\nonumber&&+\alpha t p\hat{A}_{f}\Big( |\nabla f|^{2}-f_{t}-qe^{af}-p\hat{A}\Big)-\alpha t \Delta_{\phi}p
\\\nonumber&&-2t[\alpha(a+1)-1]e^{af}\langle \nabla q,\nabla f\rangle\\\nonumber&&
-\alpha t a q e^{af}\Big( |\nabla f|^{2}-f_{t}-qe^{af}-p\hat{A}\Big)\\\nonumber&&-\alpha t\langle \nabla f,\nabla \phi_{t}\rangle
-\alpha t e^{af}\Delta_{\phi}q
-2t\big( (\alpha-1)\hat{A}+\alpha \hat{A}_{f}\big)\langle \nabla f,\nabla p\rangle
\\\label{0070}&&
-t\Big(2(\alpha-1)p\hat{A}_{f} +\alpha  p\hat{A}_{ff}+2(\alpha-1)qae^{af}+\alpha a^{2}qe^{af}\Big)|\nabla f|^{2}\\\nonumber&&-2(\alpha-1)t h(\nabla f, \nabla f)
-\frac{\alpha t n}{2\epsilon}(k_{2}+k_{3})^{2}
-3\alpha t\sqrt{n}k_{4}|\nabla f|.
\end{eqnarray}
By Young's inequality for any $\epsilon\in (0,\frac{1}{\alpha})$, we arrive at
  \begin{equation}\label{lygs}
2\alpha t h(\nabla \phi, \nabla f)\geq -2\alpha tk_{2} \langle \nabla \phi,\nabla f\rangle\geq-\frac{\alpha t k_{2}}{2\epsilon}|\nabla f|^{2}-2\alpha t k_{2}\epsilon|\nabla \phi|^{2}.
\end{equation}
Replacing  (\ref{lygs}) into (\ref{0070}) and  using  the assumptions on bounds of $Ric_{\phi}^{m-n}$ and $h$, we get the inequality (\ref{007}).
\end{proof}

\begin{proof}[Proof of theorem \ref{t1}]
Since the Ricci tensor and the evolution of the metric are bounded, then $g(t)$ is uniformly equivalent to the initial metric $g(0)$ (see \cite[Corollary 6.11]{BC}),
\begin{equation*}
e^{-2k_{2}T}g(0)\leq g(t)\leq e^{2k_{3}T}g(0).
\end{equation*}
Thus  the manifold $(M,g(t))$ is also complete for $t\in [0,T]$. Let $\psi(s)$ be a  $C^{2}$-function  on $[0,+\infty)$,
  \begin{equation*}\psi(s)=\begin{cases}1,&s\in[0,1],\\0,&s\in[2,+\infty),
\end{cases}
\end{equation*}
 and it satisfies $\psi(s)\in[0,1]$, $-c_{0}\leq\psi'(s)\leq0$, $\psi''(s)\geq-c_{1}$, and $\frac{|\psi''(s)|^{2}}{\psi(s)}\leq c_{1}$, where $c_{1}$ is an absolute constant. Let $R\geq1$ and define a   function
\begin{equation*}
\eta(x,t)=\psi(\frac{r(x,t)}{R}),
\end{equation*}
where $r(x,t)=d(x,x_{0},r)$.  Using the argument of \cite{CA, LY}, we can apply maximum principle and invoke Calabi's trick to  assume everywhere smoothness of $\eta(x,t)$ since $\psi(s)$ is in general Lipschitz. Also, we use  generalization Laplacian comparison theorem \cite{DB,  XDL, S,  JYW} to obtain inequalities of $\eta(x,t)$. Since $Ric_{\phi}^{m-n}\geq -(m-1)k_{1}$, the generalization Lpalacian comparison theorem implies that
\begin{equation*}
\Delta_{\phi}r(x)\leq (m-1)\sqrt{k_{1}}\coth(\sqrt{k_{1}}r(x))
\end{equation*}
 and
\begin{eqnarray}\label{t11}
\Delta_{\phi}\eta&=&\psi'\frac{\Delta_{\phi} r}{R}+\psi''\frac{|\nabla r|^{2}}{R^{2}}\\\nonumber&\geq&-\frac{c_{0}}{R}(m-1)\sqrt{k_{1}}\coth(\sqrt{k_{1}}r(x)) -\frac{c_{1}}{R^{2}}\\\nonumber&\geq&-\frac{c_{0}}{R}(m-1)(\sqrt{k_{1}}+\frac{2}{R})-\frac{c_{1}}{R^{2}}.
\end{eqnarray}
Also, we have
\begin{equation}\label{t12}
\frac{|\nabla \eta|^{2}}{\eta}=\frac{|\psi'|^{2}|\nabla r|^{2}}{R^{2}\psi}\leq \frac{c_{1}}{R^{2}}.
\end{equation}
Let $G=\eta F$. Fix arbitrary $T_{1}\in(0,T]$  and assume that $G$ achieves its maximum at point $(x_{0},t_{0})\in Q_{2R,T_{1}}$. If $G(x_{0},t_{0})\leq0$, then the result holds trivially  and we done. Hence, we may assume that $G(x_{0},t_{0})>0$. In this point we have
\begin{equation*}
\nabla G=0,\qquad\Delta G\leq0,\qquad \partial_{t}G\geq0.
\end{equation*}
Therefore, we conclude
\begin{equation}\label{t13}
\nabla F=-\frac{F}{\eta}\nabla \eta
\end{equation}
and
\begin{equation}\label{t14}
0\geq (\Delta_{\phi}-\partial_{t})G=F(\Delta_{\phi}-\partial_{t})\eta+\eta(\Delta_{\phi}-\partial_{t})F+2\langle\nabla\eta,\nabla F\rangle.
\end{equation}
By \cite[p. 494]{JS}, there exist  a constant $c_{2}$ such that
\begin{equation}\label{t15}
-F\eta_{t}\geq -c_{2}k_{2}F.
\end{equation}
 Replacing (\ref{t11})-(\ref{t13}) and (\ref{t15}) into (\ref{t14}) we get
\begin{eqnarray}\label{t16}
0\geq -\Big(\frac{c_{0}}{R}(m-1)(\sqrt{k_{1}}+\frac{2}{R})+\frac{3c_{1}}{R^{2}}+c_{2}k_{2} \Big)F+\eta(\Delta_{\phi}-\partial_{t})F.
\end{eqnarray}
As in \cite{LC, YZ, YY}, we set
\begin{equation*}
\mu=\frac{|\nabla f|^{2}(x_{0},t_{0})}{F(x_{0},t_{0})}\geq0,
\end{equation*}
then at point $(x_{0},t_{0})$, it follows that $|\nabla f|=\sqrt{\mu F}$,
\begin{equation*}
|\nabla f|^{2}-f_{t}-qe^{af}-p\hat{A}=\Big(\mu-\frac{t_{0}\mu-1}{t_{0}\alpha} \Big)F,
\end{equation*}

\begin{equation*}
\eta\langle \nabla f,\nabla F\rangle=-F\langle \nabla f,\nabla \eta\rangle\leq\frac{\sqrt{c_{1}}}{R}\eta^{\frac{1}{2}}F|\nabla f|,
\end{equation*}
and
\begin{equation*}
3\alpha\sqrt{n}k_{4}|\nabla f|\leq 2k_{4}|\nabla f|^{2}+\frac{9}{8}n\alpha^{2}k_{4}.
\end{equation*}
Using the above three relation, inequality (\ref{007}), and inequality (\ref{t16}) at point $(x_{0},t_{0})$, we obtain
\begin{eqnarray}\nonumber
0&\geq&\frac{2(1-\epsilon\alpha)t_{0}}{m}\eta \Big(\mu-\frac{t_{0}\mu-1}{t_{0}\alpha} \Big)^{2}F^{2}
-\frac{2\sqrt{c_{1}}}{R}\eta^{\frac{1}{2}}\mu^{\frac{1}{2}}F^{\frac{3}{2}}-\frac{\eta F}{t_{0}}
\\\nonumber&&
+\alpha t_{0}\eta p\hat{A}_{f}\Big(\mu-\frac{t_{0}\mu-1}{t_{0}\alpha} \Big)F
-2t_{0}\eta [\alpha(a+1)-1]e^{af}\langle \nabla q,\nabla f\rangle
\\\nonumber&&
-\alpha t_{0} a q e^{af}\eta\Big(\mu-\frac{t_{0}\mu-1}{t_{0}\alpha} \Big)F-\alpha t_{0}\eta\langle \nabla f,\nabla \phi_{t}\rangle
-\alpha t_{0}\eta e^{af}\Delta_{\phi}q
\\\nonumber&&
-2t_{0}\eta \big( (\alpha-1)\hat{A}+\alpha\hat{A}_{f}\big)\langle \nabla p,\nabla f\rangle-\alpha t_{0}\eta \hat{A} \Delta_{\phi}p
\\\label{107}&&
-t_{0}\eta\Big(2(\alpha-1)p\hat{A}_{f} +\alpha p \hat{A}_{ff}+2(1-\epsilon\alpha)(m-1)k_{1}+\frac{\alpha k_{2}}{2\epsilon}\\\nonumber&&+2(\alpha-1)qae^{af}+\alpha a^{2}qe^{af}+2(\alpha-1)k_{3}+2k_{4}\Big)\mu F
\\\nonumber&&
-\frac{\alpha t_{0} n}{2\epsilon}\eta(k_{2}+k_{3})^{2}
-\frac{9}{8}t_{0}\eta n\alpha^{2}k_{4}-2\alpha t_{0}\eta k_{2}\epsilon |\nabla \phi|^{2}\\\nonumber&&
-\Big(\frac{c_{0}}{R}(m-1)(\sqrt{k_{1}}+\frac{2}{R})+\frac{3c_{1}}{R^{2}}+c_{2}k_{2} \Big)F.
\end{eqnarray}
Multiply both sides of (\ref{107}) by $t_{0}\eta$. By direct computation we conclude that
\begin{eqnarray}\nonumber
0&\geq&\frac{2(1-\epsilon\alpha)}{m\alpha^{2}} \Big(1+(\alpha-1)t_{0}\mu \Big)^{2}G^{2}
-\frac{2\sqrt{c_{1}}}{R}t_{0}\mu^{\frac{1}{2}}G^{\frac{3}{2}}-\eta G
\\\nonumber&&
+t_{0}\eta p \hat{A}_{f}G
-2t_{0}^{2}\eta^{2} [\alpha(a+1)-1]e^{af}\langle \nabla q,\nabla f\rangle
\\\nonumber&&
- t_{0} a q e^{af}\eta\Big(1+(\alpha-1)t_{0}\mu \Big)G
-\alpha t_{0}^{2}\eta^{2}\langle \nabla f,\nabla \phi_{t}\rangle
-\alpha t_{0}^{2}\eta^{2} e^{af}\Delta_{\phi}q
\\\nonumber&&
-2t_{0}^{2}\eta^{2} \big( (\alpha-1)\hat{A}+\alpha\hat{A}_{f}\big)\langle \nabla p,\nabla f\rangle-\alpha t_{0}^{2}\eta^{2} \hat{A} \Delta_{\phi}p
\\\label{207}&&
-t_{0}^{2}\eta\Big((\alpha-1)p\hat{A}_{f} +\alpha  p\hat{A}_{ff}+2(1-\epsilon\alpha)(m-1)k_{1}+\frac{\alpha t_{0}k_{2}}{2\epsilon}\\\nonumber&&+2(\alpha-1)qae^{af}+\alpha a^{2}qe^{af}+2(\alpha-1)k_{3}+2k_{4}\Big)\mu G
\\\nonumber&&
-\frac{\alpha t_{0}^{2} n}{2\epsilon}\eta^{2}(k_{2}+k_{3})^{2}
-\frac{9}{8}t_{0}^{2}\eta^{2} n\alpha^{2}k_{4}-2\alpha t_{0}^{2}\eta^{2} k_{2}\epsilon |\nabla \phi|^{2}\\\nonumber&&
-t_{0}\Big(\frac{c_{0}}{R}(m-1)(\sqrt{k_{1}}+\frac{2}{R})+\frac{3c_{1}}{R^{2}}+c_{2}k_{2} \Big)G.
\end{eqnarray}
By Young's inequality, we infer
\begin{eqnarray*}
&&\frac{2\sqrt{c_{1}}}{R}\mu^{\frac{1}{2}}G^{\frac{3}{2}}\leq \frac{4(1-\epsilon\alpha)}{m\alpha^{2}}(\alpha-1)\mu G^{2}+\frac{m\alpha^{2}c_{1}G}{4(1-\epsilon\alpha)(\alpha-1)R^{2}}.
\end{eqnarray*}
The Cauchy's inequality and  Young's inequality imply that
\begin{eqnarray*}
&&2\eta^{2} [\alpha(a+1)-1]e^{af}\langle \nabla q,\nabla f\rangle\\&&\leq 2 [\alpha(a+1)-1]k|\nabla q||\nabla f|\\&&\leq
2 [\alpha(a+1)-1]k\sigma_{2}\mu^{\frac{1}{2}}G^{\frac{1}{2}}\\
&&\leq \frac{2(1-\epsilon\alpha)\delta}{m\alpha^{2}} (\alpha-1)^{2}\mu^{2}G^{2}+
\frac{3}{4}\Big(\frac{2m\alpha^{2}}{(1-\epsilon\alpha)\delta(\alpha-1)^{2}}\Big)^{\frac{1}{3}} [\alpha(a+1)-1]^{\frac{4}{3}}k^{\frac{4}{3}}\sigma_{2}^{\frac{4}{3}},
\end{eqnarray*}
and
\begin{eqnarray*}
&&2 \big( (\alpha-1)\hat{A}+\alpha\hat{A}_{f}\big)\langle \nabla p,\nabla f\rangle\\&&\leq
2 \big( (\alpha-1)\lambda_{1} +\alpha\lambda_{2}\big)| \nabla p||\nabla f|\\
&&\leq 2 \big( (\alpha-1)\lambda_{1} +\alpha\lambda_{2}\big)\gamma_{2}\mu^{\frac{1}{2}}G^{\frac{1}{2}}\\&&\leq
\frac{(1-\epsilon\alpha)(1-\delta)}{2m\alpha^{2}}  (\alpha-1)^{2}\mu^{2}G^{2}\\&&+
\frac{3}{4}\Big(\frac{m\alpha^{2}}{2(1-\epsilon\alpha)(1-\delta) (\alpha-1)^{2}}\Big)^{\frac{1}{3}} \big( (\alpha-1)\lambda_{1} +\alpha\lambda_{2}\big)^{\frac{4}{3}}\gamma_{2}^{\frac{4}{3}}.
\end{eqnarray*}
Also, we have
\begin{eqnarray*}
 \alpha\langle \nabla f,\nabla \phi_{t}\rangle&\leq&  \alpha|\nabla \Delta \phi||\nabla f|\leq
\alpha\theta_{2}\mu^{\frac{1}{2}}G^{\frac{1}{2}}\\
&\leq& \frac{(1-\epsilon\alpha)(1-\delta)}{2m\alpha^{2}}  (\alpha-1)^{2}\mu^{2}G^{2}\\&&+
\frac{3}{4}\Big(\frac{m\alpha^{2}}{2(1-\epsilon\alpha)(1-\delta) (\alpha-1)^{2}}\Big)^{\frac{1}{3}} \alpha^{\frac{4}{3}}\theta_{2}^{\frac{4}{3}},
\end{eqnarray*}
where $\delta\in(0,1)$  is an arbitrary constant. Combining the above inequalities  we conclude that
\begin{eqnarray}\nonumber
0&\geq&\frac{2(1-\epsilon\alpha)}{m\alpha^{2}}G^{2}
+\frac{(1-\epsilon\alpha)(1-\delta)}{m\alpha^{2}} (\alpha-1)^{2}t_{0}^{2}\mu^{2} G^{2}
-t_{0}\frac{m\alpha^{2}c_{1}G}{4(1-\epsilon\alpha)(\alpha-1)R^{2}}\\\nonumber&&- G-t_{0}\gamma_{1}\lambda_{2} G
-\frac{3}{4} t_{0}^{2}
\Big(\frac{2m\alpha^{2}}{(1-\epsilon\alpha)\delta (\alpha-1)^{2}}\Big)^{\frac{1}{3}} [\alpha(a+1)-1]^{\frac{42}{3}}k^{\frac{4}{3}}\sigma_{2}^{\frac{4}{3}}\\\nonumber&&-t_{0}a\sigma_{1}kG
- \frac{3}{4}t_{0}^{2}
\Big(\frac{m\alpha^{2}}{4(1-\epsilon\alpha)(1-\delta) (\alpha-1)^{2}}\Big)^{\frac{1}{3}} \alpha^{\frac{4}{3}}\theta_{2}^{\frac{4}{3}}-\alpha t_{0}^{2}k\sigma_{3}\\\nonumber&&
-\frac{3}{4}t_{0}^{2}\Big(\frac{m\alpha^{2}}{2(1-\epsilon\alpha)(1-\delta) (\alpha-1)^{2}}\Big)^{\frac{1}{3}} \big( (\alpha-1)\lambda_{1} +\alpha\lambda_{2}\big)^{\frac{4}{3}}\gamma_{2}^{\frac{4}{3}}-\alpha t_{0}^{2}\lambda_{1} \gamma_{3}
\\\label{307}&&
-t_{0}^{2}\Big((\alpha-1)( \gamma_{1}\lambda_{2}+2a\sigma_{1} k +2k_{3})+\alpha   \gamma_{1}\lambda_{3}+2(1-\epsilon\alpha)(m-1)k_{1}+\frac{\alpha k_{2}}{2\epsilon}\\\nonumber&&+\alpha a^{2}\sigma_{1} k+2k_{4}+a\sigma_{2}k\Big)\mu G
\\\nonumber&&
-\frac{\alpha t_{0}^{2} n}{2\epsilon}(k_{2}+k_{3})^{2}
-\frac{9}{8}t_{0}^{2}n\alpha^{2}k_{4}-2\alpha t_{0}^{2} k_{2}\epsilon \theta_{1}^{2}\\\nonumber&&
-t_{0}\Big(\frac{c_{0}}{R}(m-1)(\sqrt{k_{1}}+\frac{2}{R})+\frac{3c_{1}}{R^{2}}+c_{2}k_{2} \Big)G.
\end{eqnarray}
Set
\begin{eqnarray*}
C_{1}&:=&(\alpha-1)( \gamma_{1}\lambda_{2}+2a\sigma_{1} k +2k_{3})+\alpha  \gamma_{1}\lambda_{3}+2(1-\epsilon\alpha)(m-1)k_{1}\\&&+\frac{\alpha k_{2}}{2\epsilon}+\alpha a^{2}\sigma_{1} k+2k_{4}+a\sigma_{2}k.
\end{eqnarray*}
By Young's inequality, we obtain
\begin{equation}\label{407}
C_{1}\mu G\leq \frac{(1-\epsilon\alpha)(1-\delta)(\alpha-1)^{2}}{m\alpha^{2}} \mu^{2}G^{2}+\frac{m\alpha^{2}}{4(1-\epsilon\alpha)(1-\delta)(\alpha-1)^{2}}C_{1}^{2}.
\end{equation}
Plugging (\ref{407}) int (\ref{307}), we deduce
\begin{eqnarray}\nonumber
0&\geq&\frac{2(1-\epsilon\alpha)}{m\alpha^{2}}G^{2}\\\nonumber&&
-\Big[ 1+t_{0}\Big(\frac{m\alpha^{2}c_{1}}{4(1-\epsilon\alpha)(\alpha-1)R^{2}}+\gamma_{1}\lambda_{2}+a\sigma_{1}k \\\nonumber&&+\frac{c_{0}}{R}(m-1)(\sqrt{k_{1}}+\frac{2}{R})+\frac{3c_{1}}{R^{2}}+c_{2}k_{2} \Big) \Big]G
\\\nonumber&&
-\frac{3}{4} t_{0}^{2}
\Big(\frac{2m\alpha^{2}}{(1-\epsilon\alpha)\delta (\alpha-1)^{2}}\Big)^{\frac{1}{3}} [\alpha(a+1)-1]^{\frac{4}{3}}k^{\frac{4}{3}}\sigma_{2}^{\frac{4}{3}}\\\nonumber&&
- \frac{3}{4}t_{0}^{2}
\Big(\frac{m\alpha^{2}}{2(1-\epsilon\alpha)(1-\delta) (\alpha-1)^{2}}\Big)^{\frac{1}{3}} \alpha^{\frac{4}{3}}\theta_{2}^{\frac{4}{3}}-\alpha t_{0}^{2}k\sigma_{3}\\\nonumber&&
-\frac{3}{4}t_{0}^{2}\Big(\frac{m\alpha^{2}}{2(1-\epsilon\alpha)(1-\delta) (\alpha-1)^{2}}\Big)^{\frac{1}{3}} \big( (\alpha-1)\lambda_{1} +\alpha\lambda_{2}\big)^{\frac{4}{3}}\gamma_{2}^{\frac{4}{3}}-\alpha t_{0}^{2}\lambda_{1} \gamma_{3}
\\\label{507}&&
-t_{0}^{2} \frac{m\alpha^{2}}{4(1-\epsilon\alpha)(1-\delta)(\alpha-1)^{2}}C_{1}^{2}
\\\nonumber&&
-\frac{\alpha t_{0}^{2} n}{2\epsilon}(k_{2}+k_{3})^{2}
-\frac{9}{8}t_{0}^{2}n\alpha^{2}k_{4}-2\alpha t_{0}^{2} k_{2}\epsilon \theta_{1}^{2}
\end{eqnarray}
Suppose that
\begin{eqnarray*}
D_{1}&:=& 1+t_{0}\Big(\frac{m\alpha^{2}c_{1}}{4(1-\epsilon\alpha)(\alpha-1)R^{2}}+\gamma_{1}\lambda_{2}+a\sigma_{1}k \\\nonumber&&+\frac{c_{0}}{R}(m-1)(\sqrt{k_{1}}+\frac{2}{R})+\frac{3c_{1}}{R^{2}}+c_{2}k_{2} \Big) ,
\end{eqnarray*}
and
\begin{eqnarray*}\nonumber
E_{1}&:=&
\frac{3}{4}
\Big(\frac{2m\alpha^{2}}{2(1-\epsilon\alpha)\delta (\alpha-1)^{2}}\Big)^{\frac{1}{3}} [\alpha(a+1)-1]^{\frac{42}{3}}k^{\frac{4}{3}}\sigma_{2}^{\frac{4}{3}}\\\nonumber&&
+ \frac{3}{4}
\Big(\frac{m\alpha^{2}}{2(1-\epsilon\alpha)(1-\delta) (\alpha-1)^{2}}\Big)^{\frac{1}{3}} \alpha^{\frac{4}{3}}\theta_{2}^{\frac{4}{3}}+\alpha k\sigma_{3}
\\\nonumber&&
+\frac{3}{4}\Big(\frac{m\alpha^{2}}{2(1-\epsilon\alpha)(1-\delta) (\alpha-1)^{2}}\Big)^{\frac{1}{3}} \big( (\alpha-1)\lambda_{1} +\alpha\lambda_{2}\big)^{\frac{4}{3}}\gamma_{2}^{\frac{4}{3}}+\alpha \lambda_{1} \gamma_{3}
\\&&
+ \frac{m\alpha^{2}}{4(1-\epsilon\alpha)(1-\delta)(\alpha-1)^{2}}C_{1}^{2}
\\\nonumber&&
+\frac{\alpha  n}{2\epsilon}(k_{2}+k_{3})^{2}
+\frac{9}{8}n\alpha^{2}k_{4}+2\alpha  k_{2}\epsilon \theta_{1}^{2}.
\end{eqnarray*}
Then we can write (\ref{507}) as
\begin{equation*}
0\geq\frac{2(1-\epsilon\alpha)}{m\alpha^{2}}G^{2} -D_{1} G- t_{0}^{2}E_{1}.
\end{equation*}
For a positive number $\tilde{a}$ and two nonnegative  numbers $\tilde{b},\tilde{c}$,  the  quadratic inequality of the form $\tilde{a}x^{2}-\tilde{b}x-\tilde{c}\leq0$ implies that $x\leq\frac{\tilde{b}}{\tilde{a}}+\sqrt{\frac{\tilde{c}}{\tilde{a}}}$, Therefore,
 \begin{equation*}
G\leq \frac{m\alpha^{2}}{2(1-\epsilon\alpha)}D_{1}+t_{0}\Big(\frac{m\alpha^{2}}{2(1-\epsilon\alpha)}E_{1} \Big)^{\frac{1}{2}}.
\end{equation*}
To obtain the required  result on $F(x,t)$, we have $\eta(x,T_{1})=1$  whenever $d(x,x_{0},T_{1})\leq R$. Hence
\begin{eqnarray*}
&&\Big(|\nabla f|^{2}-\alpha qe^{af}-\alpha p { \hat A}(f)-\alpha f_{t}\Big)(x,T_{1})=\frac{F(x,T_{1})}{T_{1}}\leq \frac{G(x_{0},t_{0})}{T_{1}}\\&&\leq
 \frac{m\alpha^{2}}{2T_{1}(1-\epsilon\alpha)}
+
\frac{m\alpha^{2}}{2(1-\epsilon\alpha)}\Big(\frac{m\alpha^{2}c_{1}}{4(1-\epsilon\alpha)(\alpha-1)R^{2}}+\gamma_{1}\lambda_{2}+a\sigma_{1}k \\\nonumber&&+\frac{c_{0}}{R}(m-1)(\sqrt{k_{1}}+\frac{2}{R})+\frac{3c_{1}}{R^{2}}+c_{2}k_{2} \Big)
+\Big(\frac{m\alpha^{2}}{2(1-\epsilon\alpha)}E_{1} \Big)^{\frac{1}{2}}
\end{eqnarray*}
Since $T_{1}$ is arbitrary, this completes the proof.
\end{proof}
\begin{corollary}\label{c1}
Let $(M, g(0),e^{-\phi_{0}}dv)$  be a complete noncompact weighted Riemannian manifold without boundary, and let $g(t), \phi(t)$ evolve by (\ref{e2}) for $t\in [0,T]$. Let $u$ be a positive solution to (\ref{e1}) in $M$ such that $u^{a}\leq k$ for some positive constant $k$. Suppose that there exist constants $k_{1}, k_{2}, k_{3}, k_{4}$ such that
\begin{equation*}
Ric_{\phi}^{m-n}\geq -(m-1)k_{1}g,\qquad -k_{2}g\leq h\leq k_{3}g,\qquad |\nabla h|\leq k_{4},
\end{equation*}
on $M$. Then for any $\alpha>1$ and $\delta\in(0,1)$,  there exist positive constants $c_{0},c_{1},$ and $c_{2}$ such that
  \begin{equation}\label{ec1}
\frac{|\nabla u|^{2}}{u^{2}}-\alpha qu^{a}-\alpha p\frac{  A(u)}{u}-\alpha \frac{u_{t}}{u}\leq
 \frac{m\alpha^{2}}{2t(1-\epsilon\alpha)}+K_{2}
\end{equation}
on $M$,  where $\epsilon \in (0,\frac{1}{\alpha})$,
\begin{eqnarray*}
K_{2}&:=&
\frac{m\alpha^{2}}{2(1-\epsilon\alpha)}\Big(\Gamma_{1}\Lambda_{2}+a\Sigma_{1}k +c_{2}k_{2} \Big)
+\Big(\frac{m\alpha^{2}}{2(1-\epsilon\alpha)}E_{2} \Big)^{\frac{1}{2}},
\end{eqnarray*}
\begin{eqnarray*}\nonumber
E_{2}&:=&
\frac{3}{4}
\Big(\frac{2m\alpha^{2}}{(1-\epsilon\alpha)\delta (\alpha-1)^{2}}\Big)^{\frac{1}{3}} [\alpha(a+1)-1]^{\frac{4}{3}}k^{\frac{4}{3}}\Sigma_{2}^{\frac{4}{3}}\\\nonumber&&
+ \frac{3}{4}
\Big(\frac{m\alpha^{2}}{2(1-\epsilon\alpha)(1-\delta) (\alpha-1)^{2}}\Big)^{\frac{1}{3}} \alpha^{\frac{4}{3}}\Theta_{2}^{\frac{4}{3}}+\alpha k\Sigma_{3}
\\\nonumber&&
+\frac{3}{4}t_{0}^{2}\Big(\frac{m\alpha^{2}}{2(1-\epsilon\alpha)(1-\delta) (\alpha-1)^{2}}\Big)^{\frac{1}{3}} \big( (\alpha-1)\Lambda_{1} +\alpha\Lambda_{2}\big)^{\frac{4}{3}}\Gamma_{2}^{\frac{4}{3}}+\alpha t_{0}^{2}\Lambda_{1} \Gamma_{3}
\\&&
+ \frac{m\alpha^{2}}{4(1-\epsilon\alpha)(1-\delta)(\alpha-1)^{2}}C_{2}^{2}
\\\nonumber&&
+\frac{\alpha  n}{2\epsilon}(k_{2}+k_{3})^{2}
+\frac{9}{8}n\alpha^{2}k_{4}+2\alpha  k_{2}\epsilon \Theta_{1}^{2},
\end{eqnarray*}
and
\begin{eqnarray*}
C_{2}
&:=&(\alpha-1)( \Gamma_{1}\Lambda_{2}+2a\Sigma_{1} k +2k_{3})+\alpha \Gamma_{1} \Lambda_{3}+2(1-\epsilon\alpha)(m-1)k_{1}\\&&+\frac{\alpha k_{2}}{2\epsilon}+\alpha a^{2}\Sigma_{1} k+2k_{4}+a\Sigma_{2}k.
\end{eqnarray*}
\end{corollary}
\begin{proof}
Since $g(t) $ is uniformly equivalent to the initial metric $g(0)$, then $(M,g(t))$ is complete noncompact for $t\in[0,T]$.  For fixe $\delta \in (0,1)$, if $R\to +\infty$ in (\ref{10}) then we obtain inequality  (\ref{ec1}).
\end{proof}
As immediate consequence of the global gradient estimates obtained in Corollary \ref{c1}, by integrating the gradient estimates in space-time we obtain the following Harnack inequality. We first introduce the following notation.  Given $(y_{1},s_{1})\in M\times (0,T]$ and $(y_{2},s_{2})\in M\times (0,T]$ satisfying $s_{1}<s_{2}$, define
\begin{equation*}
\mathcal{J}(y_{1},s_{1},y_{2},s_{2})=\inf\int_{s_{1}}^{s_{2}}|{\zeta}'(t)|_{g(t)}^{2}dt,
\end{equation*}
and the infimum  is taken over the all smooth curves  $\zeta:[s_{1},s_{2}]\to M$ jointing $y_{1}$ and $y_{2}$.
\begin{corollary}\label{c2}
With the same assumptions in Corollary \ref{c2},  for $(y_{1},s_{1})\in M\times (0,T]$ and $(y_{2},s_{2})\in M\times (0,T]$ such that $s_{1}<s_{2}$,   we have
\begin{eqnarray*}
u(y_{1},s_{1})&\leq& u(y_{2},s_{2})\big( \frac{s_{2}}{s_{1}}\big)^{ \frac{m\alpha}{2(1-\epsilon\alpha)}}\exp\left\{ \frac{\alpha\mathcal{J}(y_{1},s_{1},y_{2},s_{2})}{4}\right.\\&&\left.+(s_{2}-s_{1})\big(  k \Sigma_{1}+ \Gamma_{1}\Lambda_{1}+\frac{1}{\alpha} K_{2}\big)\right\}.
\end{eqnarray*}
\end{corollary}
\begin{proof}
Take the geodesic path $\zeta(t)$ from $y_{1}$ to $y_{2}$ with $\zeta(s_{1})=y_{1}$ and $\zeta(s_{2})=y_{2}$. Now consider the path $(\zeta(t),t)$ in space-time. From Corollary \ref{c1}, we have the following gradient estimate
  \begin{equation}\label{e1c2}
-\partial_{t}(\ln u)\leq  k \Sigma_{1}+ \Gamma_{1}\Lambda_{1}-\frac{1}{\alpha} |\nabla (\ln u)|^{2}+
 \frac{m\alpha}{2t(1-\epsilon\alpha)}+\frac{1}{\alpha} K_{2}.
\end{equation}
Integrating this inequality along $\zeta$, we get
\begin{eqnarray*}
&&\log\frac{u(y_{1},s_{1})}{u(y_{2},s_{2})}\\
&&=-\int_{s_{1}}^{s_{2}}\frac{d}{dt}\big(\ln u(\zeta(t),t) \big)dt\\
&&=-\int_{s_{1}}^{s_{2}}\Big(\partial_{t}(\ln u)+\langle \nabla (\ln u)(\zeta(t),t) ,\dot{\zeta}(t)\rangle \Big)dt\\&&
\leq \int_{s_{1}}^{s_{2}}\left\{ k \Sigma_{1}+\Gamma_{1} \Lambda_{1}-\frac{1}{\alpha} |\nabla (\ln u)|^{2}+
 \frac{m\alpha}{2t(1-\epsilon\alpha)}+\frac{1}{\alpha} K_{2} -\langle \nabla (\ln u) ,\dot{\zeta}(t)\rangle\right\}dt\\&&
\leq\int_{s_{1}}^{s_{2}}\left\{ \frac{\alpha|\dot{\zeta}(t)|^{2}}{4}+\Big(  k \Sigma_{1}+ \Gamma_{1}\Lambda_{1}+
 \frac{m\alpha}{2t(1-\epsilon\alpha)}+\frac{1}{\alpha} K_{2}\Big)\right\}dt\\&&\leq
\frac{\alpha}{4}\int_{s_{1}}^{s_{2}}|\dot{\zeta}(t)|^{2}dt+(s_{2}-s_{1})\big(  k \Sigma_{1}+\Gamma_{1} \Lambda_{1}+\frac{1}{\alpha} K_{2}\big)+
 \frac{m\alpha}{2(1-\epsilon\alpha)}\ln\frac{s_{2}}{s_{1}},
\end{eqnarray*}
where in the computation above  we have used  (\ref{e1c2}) to obtain the inequality in the third line and used inequality $-\tilde{a}x^{2}-\tilde{b}x\leq \frac{{\tilde b}^{2}}{4\tilde{a}}$ to arrive at   the inequality in the fourth line. By exponentiation we have
\begin{equation*}
u(y_{1},s_{1})\leq u(y_{2},s_{2})\big( \frac{s_{2}}{s_{1}}\big)^{ \frac{m\alpha}{2(1-\epsilon\alpha)}}\exp\left\{ \frac{\alpha}{4}\int_{s_{1}}^{s_{2}}|\dot{\zeta}(t)|^{2}dt+(s_{2}-s_{1})\big(  k \Sigma_{1}+ \Gamma_{1}\Lambda_{1}+\frac{1}{\alpha} K_{2}\big)\right\}.
\end{equation*}
\end{proof}
 We now consider the case that the manifold $M$ is compact and without boundary, i.e., $M$ is closed, and we find  a global gradient estimate on a closed weighted Riemannian manifold.
\begin{theorem}\label{t2}
Let $(M,g(0),e^{-\phi_{0}}dv )$ be a closed  weighted Riemannian manifold, and let $g(t), \phi(t)$ evolve by (\ref{e2}) for $t\in [0,T]$. Let $u$ be a positive solution to (\ref{e1}) in $M$ such that $u^{a}\leq k$ for some positive constant $k$. Suppose that there exist constants $k_{1}, k_{2}, k_{3}, k_{4}$ such that
\begin{equation*}
Ric_{\phi}^{m-n}\geq -(m-1)k_{1}g,\qquad -k_{2}g\leq h\leq k_{3}g,\qquad |\nabla h|\leq k_{4},
\end{equation*}
on $M$. Then for any $\alpha>1$,  there exist positive constants $c_{0},c_{1},$ and $c_{2}$ such that
  \begin{equation}\label{10}
\frac{|\nabla u|^{2}}{u^{2}}-\alpha qu^{a}-\alpha p\frac{  A(u)}{u}-\alpha \frac{u_{t}}{u}\leq
 \frac{m\alpha^{2}}{2t}+K_{3}
\end{equation}
on $M$, where
\begin{eqnarray*}
K_{3}&:=&\frac{m\alpha^{2}}{2}\big(\Gamma_{1}\Lambda_{2}+ak\Sigma_{1} \big)+m\alpha^{2}(k_{2}+k_{3})\\&&+
\frac{m\alpha^{2}}{2(\alpha-1)}\Big((\alpha-1)\Gamma_{1}\Lambda_{2}+\alpha\Gamma_{1}\Lambda_{3}\Big)\\&&
+\frac{m\alpha^{2}}{2(\alpha-1)}\Big(2(m-1) k_{1}+\frac{\alpha k_{2}}{2\epsilon}+(\alpha-1)ak\Sigma_{1}+\alpha a^{2}k\Sigma_{1}+2(\alpha-1)k_{3} +1\Big)\\&&+
\sqrt{m k\Sigma_{3}}\alpha^{\frac{3}{2}}+\sqrt{2mk_{2} \Theta_{1}}\alpha^{\frac{3}{2}} \\&&+\frac{\sqrt{m}\alpha}{2}\Big(2[\alpha(a+1)-1]k\Sigma_{2}+3\alpha \sqrt{n}k_{4}+\alpha\Theta_{2}\Big).
\end{eqnarray*}

\end{theorem}
\begin{proof}
Let $f=\ln u$,  $F:=t\Big(|\nabla f|^{2}-\alpha qe^{af}-\alpha p{ \hat A}(f)-\alpha f_{t} \Big)$, and
\begin{eqnarray*}
\bar{F}(x,t):=F(x,t)-K_{3}t.
\end{eqnarray*}
For any $(x,t)\in M\times [0,T]$ if $\bar{F}(x,t)\leq \frac{m\alpha^{2}}{2}$ then  the result holds trivially. Thus, we consider  $\bar{F}(x,t)> \frac{m\alpha^{2}}{2}$ on $M\times [0,T]$.  Let $(x_{1},t_{1})$ be a point  in $M\times [0,T]$ at which $\bar{F}$ attains its maximum value. Then $\bar{F}(x_{1},t_{1})> \frac{m\alpha^{2}}{2}$. As $\bar{F}(x_{1},0)=0$ then $t_{1}>0$. Now using the maximum principle, we conclude that
\begin{equation*}
\nabla\bar{F}(x_{1},t_{1})=0,\qquad \Delta_{\phi} \bar{F}(x_{1},t_{1})\leq0,\qquad \partial_{t}\bar{F}(x_{1},t_{1})\geq0.
\end{equation*}
Therefore at point $(x_{1},t_{1})$ we deduce
\begin{equation*}
0\geq(\Delta_{\phi}-\partial_{t})\bar{F}\geq (\Delta_{\phi}-\partial_{t}){F}.
\end{equation*}
From Lemma \ref{l2} and identity
\begin{equation*}
|\nabla f|^{2}-qe^{af}-p\hat{A}-f_{t}=\frac{1}{\alpha}\frac{F}{t}+\frac{\alpha-1}{\alpha}|\nabla f|^{2}
\end{equation*}
we obtain
\begin{eqnarray*}
0&\geq& \frac{2(1-\epsilon \alpha)t_{1}}{m\alpha^{2}}\big(\frac{F}{t_{1}}\big)^{2}+\frac{4(1-\epsilon \alpha)(\alpha-1)t_{1}}{m\alpha^{2}}|\nabla f|^{2}\frac{F}{t_{1}}\\&&+ \frac{2(1-\epsilon \alpha)t_{1}}{m\alpha^{2}}(\alpha-1)^{2}t_{1}|\nabla f|^{4}-\frac{F}{t_{1}}-\Gamma_{1}\Lambda_{2}F-ak\Sigma_{1}F\\&&
-t_{1}\Big(2[\alpha(a+1)-1]k\Sigma_{2}+3\alpha \sqrt{n}k_{4}+\alpha\Theta_{2}\Big)|\nabla f|\\&&
-t_{1}\Big[(\alpha-1)\Gamma_{1}\Lambda_{2}+\alpha\Gamma_{1}\Lambda_{3}+2(1-\epsilon \alpha)(m-1)k_{1}+\frac{\alpha k_{2}}{2\epsilon}+(\alpha-1)ak\Sigma_{1}\\&&+\alpha a^{2}k\Sigma_{1}+2(\alpha-1)k_{3}
  \Big]|\nabla f|^{2}\\&&
-\alpha t_{1}k\Sigma_{3}-\frac{\alpha t_{1}n}{2\epsilon}(k_{2}+k_{3})^{2}-2\alpha t_{1}k_{2}\epsilon \Theta_{1}^{2}.
\end{eqnarray*}
Since
\begin{equation*}
\frac{F}{t_{1}}=\frac{\bar{F}}{t_{1}}+K_{3}>0
\end{equation*}
and
\begin{eqnarray*}
&&\Big(2[\alpha(a+1)-1]k\Sigma_{2}+3\alpha \sqrt{n}k_{4}+\alpha\Theta_{2}\Big)|\nabla f|\\&&\leq
|\nabla f|^{2}+\frac{1}{4}\Big(2[\alpha(a+1)-1]k\Sigma_{2}+3\alpha \sqrt{n}k_{4}+\alpha\Theta_{2}\Big)^{2},
\end{eqnarray*}
we get
\begin{eqnarray*}
0&\geq& \frac{2(1-\epsilon \alpha)t_{1}}{m\alpha^{2}}\big(\frac{F}{t_{1}}\big)^{2}-\big(1+t_{1}\Gamma_{1}\Lambda_{2}+t_{1}ak\Sigma_{1} \big)\frac{F}{t_{1}}
+ \frac{2(1-\epsilon \alpha)t_{1}}{m\alpha^{2}}(\alpha-1)^{2}t_{1}|\nabla f|^{4}\\&&
-t_{1}\Big[(\alpha-1)\Gamma_{1}\Lambda_{2}+\alpha\Gamma_{1}\Lambda_{3}+2(1-\epsilon \alpha)(m-1)k_{1}+\frac{\alpha k_{2}}{2\epsilon}\\&&+(\alpha-1)ak\Sigma_{1}+\alpha a^{2}k\Sigma_{1}+2(\alpha-1)k_{3}
  +1\Big]|\nabla f|^{2}\\&&
-\alpha t_{1}k\Sigma_{3}-\frac{\alpha t_{1}n}{2\epsilon}(k_{2}+k_{3})^{2}-2\alpha t_{1}k_{2}\epsilon \Theta_{1}^{2}\\&&-\frac{t_{1}}{4}\Big(2[\alpha(a+1)-1]k\Sigma_{2}+3\alpha \sqrt{n}k_{4}+\alpha\Theta_{2}\Big)^{2}.
\end{eqnarray*}
Set
\begin{eqnarray*}
B_{1}&:=&(\alpha-1)\Gamma_{1}\Lambda_{2}+\alpha\Gamma_{1}\Lambda_{3}+2(1-\epsilon \alpha)(m-1)k_{1}+\frac{\alpha k_{2}}{2\epsilon}\\&&+(\alpha-1)ak\Sigma_{1}+\alpha a^{2}k\Sigma_{1}+2(\alpha-1)k_{3}
  +1,
\end{eqnarray*}
and
\begin{eqnarray*}
B_{2}&:=&
\alpha k\Sigma_{3}+\frac{\alpha n}{2\epsilon}(k_{2}+k_{3})^{2}+2\alpha k_{2}\epsilon \Theta_{1}^{2}\\&&+\frac{1}{4}\Big(2[\alpha(a+1)-1]k\Sigma_{2}+3\alpha \sqrt{n}k_{4}+\alpha\Theta_{2}\Big)^{2}.
\end{eqnarray*}
Using the inequality $\tilde{a}x^{2}-\tilde{b}x\geq-\frac{{\tilde b}^{2}}{4\tilde{a}}$ holds for $\tilde{a}>0$ and $\tilde{b}\geq0$, we obtain
\begin{eqnarray*}
0&\geq& \frac{2(1-\epsilon \alpha)t_{1}}{m\alpha^{2}}\big(\frac{F}{t_{1}}\big)^{2}-\big(1+t_{1}\Gamma_{1}\Lambda_{2}+t_{1}ak\Sigma_{1} \big)\frac{F}{t_{1}}\\&&
- \frac{t_{1}m\alpha^{2}}{8(1-\epsilon \alpha)(\alpha-1)^{2}}B_{1}^{2}-t_{1}B_{2}.
\end{eqnarray*}
The inequality $\tilde{a}x^{2}-\tilde{b}x-\tilde{c}<0$ implies that $x\leq \frac{1}{2\tilde{a}}(\tilde{b}+\sqrt{{\tilde b}^{2}+4\tilde{a}\tilde{c}})$ holds  for a positive constant $\tilde{a}$ and two nonnegative constants $\tilde{b},\tilde{c}$. Then,
\begin{eqnarray*}
\frac{F}{t_{1}}&\leq&\frac{m\alpha^{2}}{4(1-\epsilon \alpha)t_{1}}\left\{1+t_{1}\Gamma_{1}\Lambda_{2}+t_{1}ak\Sigma_{1}
 +\left[\big(1+t_{1}\Gamma_{1}\Lambda_{2}+t_{1}ak\Sigma_{1} \big)^{2}\right.\right.
\\&&\left.\left.+\frac{8(1-\epsilon \alpha)t_{1}^{2}}{m\alpha^{2}}\Big(
\frac{m\alpha^{2}}{8(1-\epsilon \alpha)(\alpha-1)^{2}}B_{1}^{2}+B_{2}\Big)\right]^{\frac{1}{2}}\right\}.
\end{eqnarray*}
The inequality $\sqrt{x+y}\leq\sqrt{x}+\sqrt{y}$ for $x,y\geq0$ implies that
\begin{eqnarray*}
\frac{F}{t_{1}}&\leq&\frac{m\alpha^{2}}{4(1-\epsilon \alpha)t_{1}}\left\{1+t_{1}\Gamma_{1}\Lambda_{2}+t_{1}ak\Sigma_{1}
 +\left[\big(1+t_{1}\Gamma_{1}\Lambda_{2}+t_{1}ak\Sigma_{1} \big)^{2}\right.\right.\\&&\left.\left.
+\frac{4n(1-\epsilon \alpha)t_{1}^{2}}{m\alpha\epsilon}(k_{2}+k_{3})^{2}\right]^{\frac{1}{2}}\right\}\\&&+
\frac{m\alpha^{2}}{4(1-\epsilon \alpha)t_{1}}\left[
\frac{8(1-\epsilon \alpha)t_{1}^{2}}{m\alpha^{2}}\Big(
\frac{m\alpha^{2}}{8(1-\epsilon \alpha)(\alpha-1)^{2}}B_{1}^{2}+B_{3}\Big)\right]^{\frac{1}{2}},
\end{eqnarray*}
where $B_{3}=B_{2}-\frac{\alpha n}{2\epsilon}(k_{2}+k_{3})^{2}$. Let
$$
\epsilon=\frac{t_{1}(k_{2}+k_{3})}{1+t_{1}\Gamma_{1}\Lambda_{2}+t_{1}ak\Sigma_{1}+2t_{1}(k_{2}+k_{3})}\frac{1}{\alpha}
$$
for $k_{2}+k_{3}\neq0$. If $k_{2}+k_{3}=0$ the we can choose $\epsilon=0$. Therefore
\begin{eqnarray*}
&&\frac{m\alpha^{2}}{4(1-\epsilon \alpha)t_{1}}\left\{1+t_{1}\Gamma_{1}\Lambda_{2}+t_{1}ak\Sigma_{1}
 +\left[\big(1+t_{1}\Gamma_{1}\Lambda_{2}+t_{1}ak\Sigma_{1} \big)^{2}\right.\right.\\&&\left.\left.
+\frac{4n(1-\epsilon \alpha)t_{1}^{2}}{m\alpha\epsilon}(k_{2}+k_{3})^{2}\right]^{\frac{1}{2}}\right\}\\&&\leq \frac{m\alpha^{2}}{2t_{1}}\big(1+t_{1}\Gamma_{1}\Lambda_{2}+t_{1}ak\Sigma_{1} \big)+m\alpha^{2}(k_{2}+k_{3}),
\end{eqnarray*}
and
\begin{equation*}
\frac{m\alpha^{2}}{4(1-\epsilon \alpha)t_{1}}\left[
\frac{8(1-\epsilon \alpha)t_{1}^{2}}{m\alpha^{2}}\right]^{\frac{1}{2}}\leq \sqrt{m}\alpha.
\end{equation*}
Hence  inequality $\sqrt{x+y}\leq\sqrt{x}+\sqrt{y}$ for $x,y\geq0$ implies that
\begin{eqnarray*}
\frac{F}{t_{1}}&\leq&\frac{m\alpha^{2}}{2t_{1}}\big(1+t_{1}\Gamma_{1}\Lambda_{2}+t_{1}ak\Sigma_{1} \big)+m\alpha^{2}(k_{2}+k_{3})+
\frac{m\alpha^{2}}{2(\alpha-1)}B_{4}+B_{5},
\end{eqnarray*}
where
\begin{eqnarray*}
B_{4}&:=&(\alpha-1)\Gamma_{1}\Lambda_{2}+\alpha\Gamma_{1}\Lambda_{3}+2(m-1)k_{1}+\frac{\alpha k_{2}}{2\epsilon}+(\alpha-1)ak\Sigma_{1}+\alpha a^{2}k\Sigma_{1}\\&&+2(\alpha-1)k_{3}
  +1,
\end{eqnarray*}
and
\begin{eqnarray*}
B_{5}&:=&
\sqrt{m k\Sigma_{3}}\alpha^{\frac{3}{2}}+\sqrt{2mk_{2} \Theta_{1}}\alpha^{\frac{3}{2}} \\&&+\frac{\sqrt{m}\alpha}{2}\Big(2[\alpha(a+1)-1]k\Sigma_{2}+3\alpha \sqrt{n}k_{4}+\alpha\Theta_{2}\Big).
\end{eqnarray*}
This implies that  $\bar{F}(x_{1},t_{1})\leq \frac{m\alpha^{2}}{2}$ and this is  a contradiction. Thus the proof of  the theorem is complete.
\end{proof}
Similar to Corollary \ref{c2},  integrating the gradient estimates obtained in Theorem \ref{t2} in space-time we obtain the following Harnack inequality.
\begin{corollary}\label{c3}
With the same assumptions in Theorem \ref{t2},  for $(y_{1},s_{1})\in M\times (0,T]$ and $(y_{2},s_{2})\in M\times (0,T]$ such that $s_{1}<s_{2}$,   we have
\begin{equation*}
u(y_{1},s_{1})\leq u(y_{2},s_{2})\big( \frac{s_{2}}{s_{1}}\big)^{ \frac{m\alpha}{2}}\exp\left\{ \frac{\alpha\mathcal{J}(y_{1},s_{1},y_{2},s_{2})}{4}+(s_{2}-s_{1})\big(  k \Sigma_{1}+ \Gamma_{1}\Lambda_{1}+\frac{1}{\alpha} K_{3}\big)\right\}.
\end{equation*}
\end{corollary}
\section{Hamilton type gradient estimates}
Now we prove Hamilton type gradient estimates  for (\ref{e1})-(\ref{e2}).
\begin{theorem}\label{t3}
Let $(M,g(0),e^{-\phi_{0}}dv )$ be a complete weighted Riemannian manifold, and let $g(t), \phi(t)$ evolve by (\ref{e2}) for $t\in [0,T]$. Given $x_{0}$ and $R>0$, let $u$ be a positive solution to (\ref{e1}) in $Q_{2R,T}$ such that $\tilde{k}^{3}\leq u\leq k^{3}$ for some positive constants $k$ and $\tilde{k}$. Suppose that there exist constants $k_{1}, k_{2}$ such that
\begin{equation*}
Ric_{\phi}\geq -(n-1)k_{1}g,\qquad h\geq -k_{2}g,
\end{equation*}
on $Q_{2R,T}$. Then   there exist positive constants $c_{0},c_{1},$ and $c_{2}$ such that
\begin{eqnarray*}
\frac{|\nabla u|}{\sqrt{u}}&\leq&\frac{3}{\sqrt{\tilde k}}\Big\{
3\frac{\sqrt{c_{1}}}{R}+3^{\frac{1}{4}}(\frac{1}{4})^{\frac{1}{3}}\Big(
\frac{2}{3}k^{4}\lambda_{1}\gamma_{2}+\frac{2}{3}k^{3a+4}\sigma_{2}
\Big)^{\frac{1}{3}}\\&&+\frac{1}{\sqrt{2}}\Big[  2(n-1)k_{1} k^{2}+\gamma_{1}\lambda_{1}k^{2}+\frac{2}{3}\gamma_{1}\lambda_{2}k^{3}
+2k_{2}k+(2a+1)k^{3a+2}\sigma_{1}\\&&
+\Big(\frac{c_{0}}{R}(n-1)(\sqrt{k_{1}}+\frac{2}{R})+\frac{3c_{1}}{R^{2}}+c_{2}k_{2} \Big)k^{2}
\Big]^{\frac{1}{2}}\Big\}.
\end{eqnarray*}
\end{theorem}
Before we prove Theorem \ref{t3}, firstly we derive the following lemma.
\begin{lemma}\label{l3}
Let  $(M^{n},g,e^{-\phi}dv)$  be a weighted  Riemannian manifold, $g(t)$ evolves by (\ref{e2}) for $t\in[0,T]$ satisfies the hypotheses of Theorem \ref{t3}. If $v=u^{\frac{1}{3}}$ and $H:=v|\nabla v|^{2}$, then  we have
 \begin{eqnarray}\nonumber
(\Delta_{\phi}-\partial_{t})H&\geq&4v^{-3}H^{2}-4v^{-1}\langle \nabla v,\nabla H\rangle-2(n-1)k_{1}H-|p||\hat{A}|H\\&&-\frac{2}{3}v|p||\hat{A}_{v}|H-\frac{2}{3}v^{\frac{3}{2}}|\hat{A}||\nabla p|\sqrt{H}-2k_{2}v^{-1}H\\\nonumber&&-\frac{2}{3}v^{3a+\frac{3}{2}}|\nabla q|\sqrt{H}-(2a+1)v^{3a}|q|H.
\end{eqnarray}
\end{lemma}
\begin{proof}
Since $u=v^{3}$, we get $u_{t}=3v^{2}v_{t}$, $\nabla u=3v^{2}\nabla v$ and  $\Delta_{\phi}u=3v^{2}\Delta_{\phi}v+6v|\nabla v|^{2}$. From (\ref{e1}) we conclude that
\begin{equation}\label{e0l3}
(\Delta_{\phi}-\partial_{t})v=-2\frac{|\nabla v|^{2}}{v}+\frac{1}{3}vp\hat{A}(v)+\frac{qv^{3a+1}}{3}.
\end{equation}
By direct computation, we calculate
\begin{eqnarray}\nonumber
\partial_{t}H&=& |\nabla v|^{2}v_{t}+2v\langle \nabla v_{t},\nabla v\rangle-2h(\nabla v,\nabla v)\\\nonumber&=&
 |\nabla v|^{2}v_{t}+2v\langle \nabla \Delta_{\phi},\nabla v\rangle+8{\rm Hess v}( \nabla v,\nabla v)-4\frac{|\nabla v|^{4}}{v}-\frac{2}{3}pv\hat{A}|\nabla v|^{2}\\\label{e1l3}
&&-\frac{2}{3}v^{2}p\hat{A}_{v}|\nabla v|^{2}-\frac{2}{3}v^{2}\hat{A}\langle \nabla p,\nabla v\rangle
-\frac{2}{3}v^{3a+2}\langle \nabla q,\nabla v\rangle\\\nonumber&&
-\frac{2(3a+1)}{3}v^{3a+1}q|\nabla v|^{2}
-2h(\nabla v,\nabla v).
\end{eqnarray}
By the weighted  Bochner formula we have
\begin{eqnarray}\nonumber
\Delta_{\phi}H&=&|\nabla v|^{2}\Delta_{\phi}v+v\Delta_{\phi} |\nabla v|^{2}+2\langle \nabla |\nabla v|^{2},\nabla v\rangle\\\label{e2l3}&=&|\nabla v|^{2}\Delta_{\phi}v+4{\rm Hess v}( \nabla v,\nabla v)+2v\langle \nabla \Delta_{\phi},\nabla v\rangle\\\nonumber&&+2vRic_{\phi}( \nabla v,\nabla v)+2v|{\rm Hess v}|^{2}.
\end{eqnarray}
Combining (\ref{e1l3}) and (\ref{e2l3}), we infer
\begin{eqnarray}\nonumber
(\Delta_{\phi}-\partial_{t})H&=&|\nabla v|^{2}(\Delta_{\phi}-\partial_{t})v+4{\rm Hess v}( \nabla v,\nabla v)+2vRic_{\phi}( \nabla v,\nabla v)\\\label{e3l3}&&+2v|{\rm Hess v}|^{2}+4\frac{|\nabla v|^{4}}{v}
+\frac{2}{3}pv\hat{A}|\nabla v|^{2}\\\nonumber
&&+\frac{2}{3}v^{2}p\hat{A}_{v}|\nabla v|^{2}+\frac{2}{3}v^{2}\hat{A}\langle \nabla p,\nabla v\rangle
+\frac{2}{3}v^{3a+2}\langle \nabla q,\nabla v\rangle\\\nonumber&&
+\frac{2(3a+1)}{3}v^{3a+1}q|\nabla v|^{2}
+2h(\nabla v,\nabla v)
\end{eqnarray}
Substituting (\ref{e0l3}) into (\ref{e3l3}) we get
\begin{eqnarray*}
(\Delta_{\phi}-\partial_{t})H&=&
2v|{\rm Hess v}|^{2}-4{\rm Hess v}( \nabla v,\nabla v)+2\frac{|\nabla v|^{4}}{v}+2vRic_{\phi}( \nabla v,\nabla v)\\&&+pv\hat{A}|\nabla v|^{2}
+\frac{2}{3}v^{2}p\hat{A}_{v}|\nabla v|^{2}+\frac{2}{3}v^{2}\hat{A}\langle \nabla p,\nabla v\rangle
+\frac{2}{3}v^{3a+2}\langle \nabla q,\nabla v\rangle\\\nonumber&&
+(2a+1)v^{3a+1}q|\nabla v|^{2}
+2h(\nabla v,\nabla v)
\end{eqnarray*}
As
\begin{equation*}
2v|{\rm Hess v}|^{2}+4{\rm Hess v}( \nabla v,\nabla v)+2\frac{|\nabla v|^{4}}{v}=2v\big|{\rm Hess}v+\frac{\nabla v\otimes\nabla v}{v}\big|^{2}\geq0
\end{equation*}
we obtain
\begin{eqnarray*}
(\Delta_{\phi}-\partial_{t})H&\geq&
-8{\rm Hess v}( \nabla v,\nabla v)+2vRic_{\phi}( \nabla v,\nabla v)+pv\hat{A}|\nabla v|^{2}
+\frac{2}{3}v^{2}p\hat{A}_{v}|\nabla v|^{2}\\&&+\frac{2}{3}v^{2}\hat{A}\langle \nabla p,\nabla v\rangle
+\frac{2}{3}v^{3a+2}\langle \nabla q,\nabla v\rangle+(2a+1)v^{3a+1}q|\nabla v|^{2}\\\nonumber&&
+2h(\nabla v,\nabla v)\\&\geq&
4v^{-3}H^{2}-4v^{-1}\langle \nabla v,\nabla H\rangle+2vRic_{\phi}( \nabla v,\nabla v)-|p||\hat{A}|H\\&&-\frac{2}{3}v|p||\hat{A}_{v}|H-\frac{2}{3}v^{\frac{3}{2}}|\hat{A}||\nabla p|\sqrt{H}-\frac{2}{3}v^{3a+\frac{3}{2}}|\nabla q|\sqrt{H}\\\nonumber&&-(2a+1)v^{3a}|q|H+2h(\nabla v,\nabla v).
\end{eqnarray*}
The assumptions $h\geq-k_{2}g$ and $Ric_{\phi}\geq -(n-1)k_{1}g$ complete the proof of Lemma.
\end{proof}
\begin{proof}[Proof of Theorem \ref{t3}]
Choosing $\psi$ and $\eta$  as in the proof of Theorem \ref{t1}. For any $T_{1}\in(0,T]$, let $(x_{2},t_{2})\in Q_{2R,T_{1}}$ be a point  where $\mathcal{G}(x,t)=\eta(x,t)H(x,t)$ achieve  its maximum, and without loss of generality we can assume $\mathcal{G}(x_{2},t_{2})>0$, and then  $H(x_{2},t_{2})>0$. By Lemma \ref{l3} and a similar argument as in the proof of Theorem \ref{t1},  at point $(x_{2},t_{2})$ we have
\begin{eqnarray*}
0&\geq&  (\Delta_{\phi}-\partial_{t})\mathcal{G}= (\Delta_{\phi}-\partial_{t})(\eta H)\\&=&
H(\Delta_{\phi}-\partial_{t})\eta+\eta(\Delta_{\phi}-\partial_{t})H+2\langle\nabla\eta,\nabla H\rangle\\&\geq&
\eta(\Delta_{\phi}-\partial_{t})H
 -\Big(\frac{c_{0}}{R}(n-1)(\sqrt{k_{1}}+\frac{2}{R})+\frac{3c_{1}}{R^{2}}+c_{2}k_{2} \Big)H\\&\geq&
4v^{-3}\eta H^{2}-4\frac{\sqrt{c_{1}}}{R}\eta^{\frac{1}{2}}v^{-1} H |\nabla v|-2(n-1)k_{1}\mathcal{G}-\gamma_{1}\lambda_{1}\mathcal{G}-\frac{2}{3}v\gamma_{1}\lambda_{2}\mathcal{G}\\&&
-\frac{2}{3}v^{\frac{3}{2}}\eta \lambda_{1}\gamma_{2}\sqrt{H}-2k_{2}v^{-1}\mathcal{G}
-\frac{2}{3}v^{3a+\frac{3}{2}}\sigma_{2}\eta\sqrt{H}-(2a+1)v^{3a}\sigma_{1}\mathcal{G}\\&&
 -\Big(\frac{c_{0}}{R}(n-1)(\sqrt{k_{1}}+\frac{2}{R})+\frac{3c_{1}}{R^{2}}+c_{2}k_{2} \Big)H.
\end{eqnarray*}
Multiply both sides of above inequality by $\eta v^{3}$. An elementary calculation  implies that
\begin{eqnarray}\nonumber
0&\geq&
4v^{2}\mathcal{G}^{2}-4\frac{\sqrt{c_{1}}}{R}v^{\frac{3}{2}} \mathcal{G}^{\frac{3}{2}} -2(n-1)k_{1}\eta v^{3}\mathcal{G}-\gamma_{1}\lambda_{1}v^{3}\eta\mathcal{G}-\frac{2}{3}\gamma_{1}\lambda_{2}v^{4}\eta\mathcal{G}\\\nonumber&&
-\frac{2}{3}v^{\frac{9}{2}}\eta^{\frac{3}{2}} \lambda_{1}\gamma_{2}\sqrt{\mathcal{G}}-2k_{2}v^{2}\eta\mathcal{G}
-\frac{2}{3}v^{3a+\frac{9}{2}}\sigma_{2}\eta^{\frac{3}{2}}\sqrt{\mathcal{G}}-(2a+1)v^{3a+3}\sigma_{1}\eta\mathcal{G}\\\nonumber&&
 -\Big(\frac{c_{0}}{R}(n-1)(\sqrt{k_{1}}+\frac{2}{R})+\frac{3c_{1}}{R^{2}}+c_{2}k_{2} \Big) v^{3}\mathcal{G}\\\nonumber&= &
4v^{2}\mathcal{G}^{2}-4\frac{\sqrt{c_{1}}}{R}v^{\frac{3}{2}} \mathcal{G}^{\frac{3}{2}}\\\label{e5l3}&&
-\Big[  2(n-1)k_{1} v^{2}+\gamma_{1}\lambda_{1}v^{2}+\frac{2}{3}\gamma_{1}\lambda_{2}v^{3}
+2k_{2}v+(2a+1)v^{3a+2}\sigma_{1}\Big] v\eta \mathcal{G}\\\nonumber&&
-\Big(\frac{c_{0}}{R}(n-1)(\sqrt{k_{1}}+\frac{2}{R})+\frac{3c_{1}}{R^{2}}+c_{2}k_{2} \Big)v^{3}\mathcal{G}
\\\nonumber&&
-\Big(
\frac{2}{3}v^{4}\lambda_{1}\gamma_{2}+\frac{2}{3}v^{3a+4}\sigma_{2}
\Big)\eta^{\frac{31}{2}} \sqrt{v\mathcal{G}},
\end{eqnarray}
 at point $(x_{2},t_{2})$. By Young's inequality and as $\tilde{k}^{3}\leq u\leq k^{3}$, we infer
\begin{eqnarray*}
4\frac{\sqrt{c_{1}}}{R}v^{\frac{3}{2}} \mathcal{G}^{\frac{3}{2}}\leq v^{2}\mathcal{G}^{2}+81\frac{c_{1}^{2}}{R^{4}},
\end{eqnarray*}
and
\begin{eqnarray*}
&&\Big[  2(n-1)k_{1} v^{2}+\gamma_{1}\lambda_{1}v^{2}+\frac{2}{3}\gamma_{1}\lambda_{2}v^{3}
+2k_{2}v+(2a+1)v^{3a+2}\sigma_{1}\Big] v\eta \mathcal{G}\\&&
-\Big(\frac{c_{0}}{R}(n-1)(\sqrt{k_{1}}+\frac{2}{R})+\frac{3c_{1}}{R^{2}}+c_{2}k_{2} \Big)v^{3}\mathcal{G}
\\&&\leq v^{2}\mathcal{G}^{2}+\frac{1}{4}\Big[  2(n-1)k_{1} v^{2}+\gamma_{1}\lambda_{1}v^{2}+\frac{2}{3}\gamma_{1}\lambda_{2}v^{3}
+2k_{2}v+(2a+1)v^{3a+2}\sigma_{1}\\&&
+\Big(\frac{c_{0}}{R}(n-1)(\sqrt{k_{1}}+\frac{2}{R})+\frac{3c_{1}}{R^{2}}+c_{2}k_{2} \Big)v^{2}
\Big]^{2}
\\&&\leq v^{2}\mathcal{G}^{2}+\frac{1}{4}\Big[  2(n-1)k_{1} k^{2}+\gamma_{1}\lambda_{1}k^{2}+\frac{2}{3}\gamma_{1}\lambda_{2}k^{3}
+2k_{2}k+(2a+1)k^{3a+2}\sigma_{1}\\&&
+\Big(\frac{c_{0}}{R}(n-1)(\sqrt{k_{1}}+\frac{2}{R})+\frac{3c_{1}}{R^{2}}+c_{2}k_{2} \Big)k^{2}
\Big]^{2}.
\end{eqnarray*}
Using Young's inequality again, we can write
\begin{eqnarray*}
&&\Big(
\frac{2}{3}v^{4}\lambda_{1}\gamma_{2}+\frac{2}{3}v^{3a+4}\sigma_{2}
\Big)\eta^{\frac{3}{2}} \sqrt{v\mathcal{G}}
\\&&\leq
v^{2}\mathcal{G}^{2}+3(\frac{1}{4})^{\frac{4}{3}}\Big(
\frac{2}{3}v^{4}\lambda_{1}\gamma_{2}+\frac{2}{3}v^{3a+4}\sigma_{2}
\Big)^{\frac{4}{3}}
\\&&\leq
v^{2}\mathcal{G}^{2}+3(\frac{1}{4})^{\frac{4}{3}}\Big(
\frac{2}{3}k^{4}\lambda_{1}\gamma_{2}+\frac{2}{3}k^{3a+4}\sigma_{2}
\Big)^{\frac{4}{3}}.
\end{eqnarray*}
 Substituting the above three inequalities into (\ref{e5l3}), we arrive at
\begin{eqnarray*}
v^{2}\mathcal{G}^{2}&\leq&
81\frac{c_{1}^{2}}{R^{4}}+3(\frac{1}{4})^{\frac{4}{3}}\Big(
\frac{2}{3}k^{4}\lambda_{1}\gamma_{2}+\frac{2}{3}k^{3a+4}\sigma_{2}
\Big)^{\frac{4}{3}}\\&&+\frac{1}{4}\Big[  2(n-1)k_{1} k^{2}+\gamma_{1}\lambda_{1}k^{2}+\frac{2}{3}\gamma_{1}\lambda_{2}k^{3}
+2k_{2}k+(2a+1)k^{3a+2}\sigma_{1}\\&&
+\Big(\frac{c_{0}}{R}(n-1)(\sqrt{k_{1}}+\frac{2}{R})+\frac{3c_{1}}{R^{2}}+c_{2}k_{2} \Big)k^{2}
\Big]^{2},
\end{eqnarray*}
 at point $(x_{2},t_{2})$. Since $\tilde{k}^{3}\leq u$, then $\tilde{k}^{2}H^{2}\leq v^{2}\mathcal{G}^{2}$. Therefore
\begin{eqnarray*}
\big(\frac{|\nabla u|^{2}}{9u} \big)^{2}&\leq&\frac{1}{{\tilde k}^{2}}\Big\{
81\frac{c_{1}^{2}}{R^{4}}+3(\frac{1}{4})^{\frac{4}{3}}\Big(
\frac{2}{3}k^{4}\lambda_{1}\gamma_{2}+\frac{2}{3}k^{3a+4}\sigma_{2}
\Big)^{\frac{4}{3}}\\&&+\frac{1}{4}\Big[  2(n-1)k_{1} k^{2}+\gamma_{1}\lambda_{1}k^{2}+\frac{2}{3}\gamma_{1}\lambda_{2}k^{3}
+2k_{2}k+(2a+1)k^{3a+2}\sigma_{1}\\&&
+\Big(\frac{c_{0}}{R}(n-1)(\sqrt{k_{1}}+\frac{2}{R})+\frac{3c_{1}}{R^{2}}+c_{2}k_{2} \Big)k^{2}
\Big]^{2}\Big\},
\end{eqnarray*}
 at point $(x_{2},t_{2})$.
 Hence  inequality $\sqrt{x+y}\leq\sqrt{x}+\sqrt{y}$ for $x,y\geq0$ implies that
\begin{eqnarray*}
\frac{|\nabla u|}{\sqrt{u}} &\leq&\frac{3}{\sqrt{\tilde k}}\Big\{
3\frac{\sqrt{c_{1}}}{R}+3^{\frac{1}{4}}(\frac{1}{4})^{\frac{1}{3}}\Big(
\frac{2}{3}k^{4}\lambda_{1}\gamma_{2}+\frac{2}{3}k^{3a+4}\sigma_{2}
\Big)^{\frac{1}{3}}\\&&+\frac{1}{\sqrt{2}}\Big[  2(n-1)k_{1} k^{2}+\gamma_{1}\lambda_{1}k^{2}+\frac{2}{3}\gamma_{1}\lambda_{2}k^{3}
+2k_{2}k+(2a+1)k^{3a+2}\sigma_{1}\\&&
+\Big(\frac{c_{0}}{R}(n-1)(\sqrt{k_{1}}+\frac{2}{R})+\frac{3c_{1}}{R^{2}}+c_{2}k_{2} \Big)k^{2}
\Big]^{\frac{1}{2}}\Big\},
\end{eqnarray*}
 at point $(x_{2},t_{2})$.
This completes the proof of theorem.
\end{proof}
Similar to Corollary \ref{c1}, when $(M, g(0),e^{-\phi_{0}}dv)$  be a complete noncompact weighted Riemannian manifold without boundary and  $g(t)$ evolve by (\ref{e2}), we can conclude  a global gradient estimate  from Theorem \ref{t3} by taking $R\to +\infty$ as follows.
\begin{corollary}\label{c4}
Let $(M, g(0),e^{-\phi_{0}}dv)$  be a complete noncompact weighted Riemannian manifold without boundary, and let $g(t), \phi(t)$ evolve by (\ref{e2}) for $t\in [0,T]$. Let $u$ be a positive solution to (\ref{e1}) in $M$ such that $\tilde{k}^{3}\leq u\leq k^{3}$ for some positive constants $k$ and $\tilde{k}$. Suppose that there exist constants $k_{1}, k_{2}$ such that
\begin{equation*}
Ric_{\phi}\geq -(n-1)k_{1}g,\qquad h\geq -k_{2}g.
\end{equation*}
on $M$. Then  there exist positive constants $c_{0},c_{1},$ and $c_{2}$ such that
\begin{eqnarray*}
\frac{|\nabla u|}{\sqrt{u}} &\leq&\frac{3}{\sqrt{\tilde k}}\Big\{
3^{\frac{1}{4}}(\frac{1}{4})^{\frac{1}{3}}\Big(
\frac{2}{3}k^{4}\Lambda_{1}\Gamma_{2}+\frac{2}{3}k^{3a+4}\Sigma_{2}
\Big)^{\frac{1}{3}}\\&&+\frac{1}{\sqrt{2}}\Big[  2(n-1)k_{1} k^{2}+\Gamma_{1}\Lambda_{1}k^{2}+\frac{2}{3}\Gamma_{1}\Lambda_{2}k^{3}
\\&&+2k_{2}k+(2a+1)k^{3a+2}\Sigma_{1}
+c_{2}k_{2}k^{2}
\Big]^{\frac{1}{2}}\Big\}.
\end{eqnarray*}
\end{corollary}
\section{Souplet-Zhang type gradient estimate}
In this section we obtain a Souplet-Zhang type gradient estimate for (\ref{e1}) under (\ref{e2}).
\begin{theorem}\label{t4}
 Let $(M,g(0),e^{-\phi_{0}}dv )$ be a complete weighted Riemannian manifold, and let $g(t), \phi(t)$ evolve by (\ref{e2}) for $t\in [0,T]$. Given $x_{0}$ and $R>0$, let $u$ be a positive solution to (\ref{e1}) in $Q_{2R,T}$ such that $\tilde{k}\leq u\leq k$ for some positive constants $k$ and $\tilde{k}$. Suppose that there exist constants $k_{1}, k_{2}$ such that
\begin{equation*}
Ric_{\phi}\geq -(n-1)k_{1}g,\qquad h\geq -k_{2}g,
\end{equation*}
on $Q_{2R,T}$. Then   there exist positive constants $c_{0},c_{1},$ and $c_{2}$ such that
\begin{eqnarray}\label{et4}
\frac{|\nabla u|}{u}&\leq&(1+\ln \frac{k}{\tilde{k}})\Big[
6^{\frac{1}{4}}\big(\ln \frac{k}{\tilde{k}}\big)\frac{\sqrt{c_{1}}}{R}
+(\frac{3}{4})^{\frac{1}{4}}\Big(2\lambda_{1} \gamma_{2}+2k^{a}\sigma_{2}\Big)^{\frac{1}{3}}\\\nonumber&&
+\sqrt{2}\Big(\gamma_{1}\lambda_{2}+(a+1)\sigma_{1}k^{a}+(n-1)k_{1}+k_{2}+\gamma_{1}\lambda_{1}
\Big)^{\frac{1}{2}}\\\nonumber&&
+\Big(\frac{c_{0}}{R}(n-1)(\sqrt{k_{1}}+\frac{2}{R})+\frac{3c_{1}}{R^{2}}+c_{2}k_{2} \Big)^{\frac{1}{2}}\Big].
\end{eqnarray}
\end{theorem}
Before we prove Theorem \ref{t4}, firstly we derive the following lemma.
\begin{lemma}\label{l4}
Let  $(M^{n},g(t),e^{-\phi}dv)$  be a weighted  Riemannian manifold, $g(t)$ evolves by (\ref{e2}) for $t\in[0,T]$ satisfies the hypotheses of Theorem \ref{t4}. If $f=\ln u$, $\rho:=1+\ln k $, and $w:=|\nabla \ln(\rho-f)|^{2}=\frac{|\nabla u|^{2}}{(\rho-f)^{2}}$. Then  we have
 \begin{eqnarray*}\nonumber
(\Delta_{\phi}-\partial_{t})w&\geq&\frac{2(f-\ln k)}{\rho-f}\langle \nabla w, \nabla f\rangle+2(\rho-f)w^{2}\\&&
-\frac{2}{\rho-f}\Big( |\hat{A}|\nabla p|+e^{af}|\nabla q|\Big)\sqrt{w}\\&&
+2\Big(p\hat{A}_{f}+aqe^{af}-(n-1)k_{1}-k_{2}+\frac{qe^{af}}{\rho-f}+\frac{p\hat{A}}{\rho-f} \Big)w.
\end{eqnarray*}
\end{lemma}
\begin{proof}
By the weighted Bochner formula, we get
\begin{eqnarray}\label{e2l4}
\Delta_{\phi}w&=&\frac{2\langle \nabla \Delta_{\phi}f,\nabla f\rangle}{(\rho-f)^{2}}+\frac{2Ric_{\phi}(\nabla f,\nabla f)}{(\rho-f)^{2}}+\frac{2|{\rm Hess}f|^{2}}{(\rho-f)^{2}}\\\nonumber&&
+\frac{8{\rm Hess}(\nabla f, \nabla f)}{(\rho-f)^{3}}+\frac{2|\nabla f|^{2}\Delta_{\phi}f}{(\rho-f)^{3}}+\frac{6|\nabla f|^{4}}{(\rho-f)^{4}}.
\end{eqnarray}
On the other hand
\begin{eqnarray}\label{e3l4}
\partial_{t}w&=&\frac{2\langle \nabla f_{t},\nabla f\rangle}{(\rho-f)^{2}}+\frac{2|\nabla f|^{2}f_{t}}{(\rho-f)^{3}}-\frac{2h(\nabla f,\nabla f)}{(\rho-f)^{2}}.
\end{eqnarray}
Combining the above two equalities, we infer
\begin{eqnarray*}
(\Delta_{\phi}-\partial_{t})w&=&\frac{2\langle \nabla (\Delta_{\phi}-\partial_{t})f,\nabla f\rangle}{(\rho-f)^{2}}+\frac{2Ric_{\phi}(\nabla f,\nabla f)}{(\rho-f)^{2}}+\frac{2|{\rm Hess}f|^{2}}{(\rho-f)^{2}}\\\nonumber&&
+\frac{8{\rm Hess}(\nabla f, \nabla f)}{(\rho-f)^{3}}+\frac{2|\nabla f|^{2}(\Delta_{\phi}-\partial_{t}) f}{(\rho-f)^{3}}+\frac{6|\nabla f|^{4}}{(\rho-f)^{4}}
+\frac{2h(\nabla f,\nabla f)}{(\rho-f)^{2}}.
\end{eqnarray*}
Plugging (\ref{e3}) in above equality, we conclude
\begin{eqnarray*}
(\Delta_{\phi}-\partial_{t})w&=&\frac{2|{\rm Hess}f|^{2}}{(\rho-f)^{2}}+\frac{8{\rm Hess}(\nabla f, \nabla f)}{(\rho-f)^{3}}+\frac{6|\nabla f|^{4}}{(\rho-f)^{4}}-\frac{4{\rm Hess}(\nabla f, \nabla f)}{(\rho-f)^{2}}\\&&
-\frac{2|\nabla f|^{4}}{(\rho-f)^{3}}
+\frac{2\hat{A}\langle \nabla p,\nabla f\rangle}{(\rho-f)^{2}}
+\frac{2p\hat{A}_{f}|\nabla f|^{2}}{(\rho-f)^{2}}
+\frac{2e^{af}\langle \nabla q,\nabla f\rangle}{(\rho-f)^{2}}\\&&
+\frac{2aqe^{af}|\nabla f|^{2}}{(\rho-f)^{2}}
+\frac{2Ric_{\phi}(\nabla f,\nabla f)}{(\rho-f)^{2}}
+\frac{2h(\nabla f,\nabla f)}{(\rho-f)^{2}}
+\frac{2aqe^{af}|\nabla f|^{2}}{(\rho-f)^{3}}\\&&
+\frac{2p\hat{A}|\nabla f|^{2}}{(\rho-f)^{3}}.
\end{eqnarray*}
Since
\begin{equation*}
\frac{|{\rm Hess}f|^{2}}{(\rho-f)^{2}}+\frac{2{\rm Hess}(\nabla f, \nabla f)}{(\rho-f)^{3}}+\frac{|\nabla f|^{4}}{(\rho-f)^{4}}=\frac{1}{(\rho-f)^{2}}\Big|{\rm Hess}+\frac{df\otimes df}{\rho-f} \Big|^{2}\geq0
\end{equation*}
and
\begin{equation*}
\frac{2{\rm Hess}(\nabla f, \nabla f)}{(\rho-f)^{2}}+\frac{2|\nabla f|^{4}}{(\rho-f)^{3}}=\langle \nabla w, \nabla f\rangle,
\end{equation*}
we deduce
\begin{eqnarray*}
(\Delta_{\phi}-\partial_{t})w&\geq&\frac{2\langle \nabla w, \nabla f\rangle}{\rho-f}-2\langle \nabla w, \nabla f\rangle+\frac{2|\nabla f|^{4}}{(\rho-f)^{3}}
+\frac{2\hat{A}\langle \nabla p,\nabla f\rangle}{(\rho-f)^{2}}
+\frac{2p\hat{A}_{f}|\nabla f|^{2}}{(\rho-f)^{2}}\\&&
+\frac{2e^{af}\langle \nabla q,\nabla f\rangle}{(\rho-f)^{2}}
+\frac{2aqe^{af}|\nabla f|^{2}}{(\rho-f)^{2}}
+\frac{2Ric_{\phi}(\nabla f,\nabla f)}{(\rho-f)^{2}}
+\frac{2h(\nabla f,\nabla f)}{(\rho-f)^{2}}\\&&
+\frac{2aqe^{af}|\nabla f|^{2}}{(\rho-f)^{3}}
+\frac{2p\hat{A}|\nabla f|^{2}}{(\rho-f)^{3}}\\
&\geq&
\frac{2(f-\ln k)}{\rho-f}\langle \nabla w, \nabla f\rangle+2(\rho-f)w^{2}
-\frac{2|\hat{A}|| \nabla p|\sqrt{w}}{\rho-f}
+2p\hat{A}_{f}w\\&&
-\frac{2e^{af}| \nabla q|sqrt{w}}{\rho-f}
+2aqe^{af}w-2(n-1)k_{1}w-2k_{2}w\\&&
+\frac{2aqe^{af}w}{\rho-f}
+\frac{2p\hat{A}w}{\rho-f}.
\end{eqnarray*}
This completes the proof of Lemma.
\end{proof}
\begin{proof}[Proof of Theorem \ref{t4}]
Choosing $\psi$ and $\eta$  as in the proof of Theorems \ref{t1} and \ref{t3}. For any $T_{1}\in(0,T]$, let $(x_{3},t_{3})\in Q_{2R,T_{1}}$ be a point  where $\mathcal{H}(x,t)=\eta(x,t)w(x,t)$ achieve  its maximum, and without loss of generality we can assume $\mathcal{H}(x_{3},t_{3})>0$, and then  $w(x_{3},t_{3})>0$. By Lemma \ref{l4} and a similar argument as in the proof of Theorem \ref{t1},  at point $(x_{3},t_{3})$ we have
\begin{eqnarray*}
0&\geq&  (\Delta_{\phi}-\partial_{t})\mathcal{H}= (\Delta_{\phi}-\partial_{t})(\eta w)\\&=&
w(\Delta_{\phi}-\partial_{t})\eta+\eta(\Delta_{\phi}-\partial_{t})w+2\langle\nabla\eta,\nabla w\rangle\\&\geq&
\eta(\Delta_{\phi}-\partial_{t})w
 -\Big(\frac{c_{0}}{R}(n-1)(\sqrt{k_{1}}+\frac{2}{R})+\frac{3c_{1}}{R^{2}}+c_{2}k_{2} \Big)w.
\end{eqnarray*}
By multiplying both sides of above inequality by $\eta$ we can write
 \begin{eqnarray}\nonumber
0&\geq&\frac{2(f-\ln k)}{\rho-f}\eta^{2}\langle \nabla w, \nabla f\rangle+2(\rho-f)\mathcal{H}^{2}\\\label{e1t4}&&
-\frac{2}{\rho-f}\eta^{\frac{3}{2}}\Big( |\hat{A}|\nabla p|+e^{af}|\nabla q|\Big)\sqrt{\mathcal{H}}\\\nonumber&&
-2\Big(|p| |\hat{A}_{f}|+a|q|e^{af}+(n-1)k_{1}+k_{2}+\frac{|q|e^{af}}{\rho-f}+\frac{|p||\hat{A}|}{\rho-f} \Big)\eta \mathcal{H}\\\nonumber&& -\Big(\frac{c_{0}}{R}(n-1)(\sqrt{k_{1}}+\frac{2}{R})+\frac{3c_{1}}{R^{2}}+c_{2}k_{2} \Big)\mathcal{H},
\end{eqnarray}
  at point $(x_{3},t_{3})$. Now we deal with each term  of the right-hand side of above inequality. Using Young's inequality  we have
\begin{eqnarray*}
\frac{2(\ln k-f)}{\rho-f}\eta^{2}\langle \nabla w, \nabla f\rangle&=&\frac{2(f-\ln k)}{\rho-f}\mathcal {H}\langle \nabla \eta, \nabla f\rangle\leq 2\ln( \frac{k}{\tilde{k}})\frac{\sqrt{c_{1}}}{R}\mathcal {H}^{\frac{3}{2}}\\&\leq& \frac{1}{4}\mathcal {H}^{2}+6\big(\ln \frac{k}{\tilde{k}}\big)^{4}\frac{c_{1}^{2}}{R^{4}},
\end{eqnarray*}
\begin{eqnarray*}\nonumber
\frac{2}{\rho-f}\eta^{\frac{3}{2}}\Big( |\hat{A}|\nabla p|+e^{af}|\nabla q|\Big)\sqrt{\mathcal{H}}&\leq&
2\Big(\lambda_{1} \gamma_{2}+k^{a}\sigma_{2}\Big)\sqrt{\mathcal{H}}\\&\leq&
\frac{1}{4}\mathcal {H}^{2}+ \frac{3}{4}\Big(2\lambda_{1} \gamma_{2}+2k^{a}\sigma_{2}\Big)^{\frac{4}{3}},
\end{eqnarray*}
and
\begin{eqnarray*}
&&2\Big(|p| |\hat{A}_{f}|+a|q|e^{af}+(n-1)k_{1}+k_{2}+\frac{|q|e^{af}}{\rho-f}+\frac{|p||\hat{A}|}{\rho-f} \Big)\eta \mathcal{H}\\&&\leq
2\Big(\gamma_{1}\lambda_{2}+(a+1)\sigma_{1}k^{a}+(n-1)k_{1}+k_{2}+\gamma_{1}\lambda_{1}
\Big)\eta \mathcal{H}\\&&\leq
\frac{1}{4}\mathcal {H}^{2}+ 4\Big(\gamma_{1}\lambda_{2}+(a+1)\sigma_{1}k^{a}+(n-1)k_{1}+k_{2}+\gamma_{1}\lambda_{1}
\Big)^{2}.
\end{eqnarray*}
Using Young inequality again we get
\begin{eqnarray*}
&&\Big(\frac{c_{0}}{R}(n-1)(\sqrt{k_{1}}+\frac{2}{R})+\frac{3c_{1}}{R^{2}}+c_{2}k_{2} \Big)\mathcal{H}
\\&&\leq
\frac{1}{4}\mathcal {H}^{2}+\Big(\frac{c_{0}}{R}(n-1)(\sqrt{k_{1}}+\frac{2}{R})+\frac{3c_{1}}{R^{2}}+c_{2}k_{2} \Big)^{2}.
\end{eqnarray*}
Substituting the above four inequality into (\ref{e1t4}), we obtain
\begin{eqnarray*}
0&\geq&2(\rho-f)\mathcal{H}^{2}-\mathcal {H}^{2}
-6\big(\ln \frac{k}{\tilde{k}}\big)^{4}\frac{c_{1}^{2}}{R^{4}}
-\frac{3}{4}\Big(2\lambda_{1} \gamma_{2}+2k^{a}\sigma_{2}\Big)^{\frac{4}{3}}\\&&
-4\Big(\gamma_{1}\lambda_{2}+(a+1)\sigma_{1}k^{a}+(n-1)k_{1}+k_{2}+\gamma_{1}\lambda_{1}
\Big)^{2}\\&&
-\Big(\frac{c_{0}}{R}(n-1)(\sqrt{k_{1}}+\frac{2}{R})+\frac{3c_{1}}{R^{2}}+c_{2}k_{2} \Big)^{2},
\end{eqnarray*}
  at point $(x_{3},t_{3})$. Since $2\mathcal{H}^{2}\leq 2(\rho-f)\mathcal{H}^{2}$ we get
\begin{eqnarray*}
\mathcal{H}^{2}&\leq&
6\big(\ln \frac{k}{\tilde{k}}\big)^{4}\frac{c_{1}^{2}}{R^{4}}
+\frac{3}{4}\Big(2\lambda_{1} \gamma_{2}+2k^{a}\sigma_{2}\Big)^{\frac{4}{3}}\\&&
+4\Big(\gamma_{1}\lambda_{2}+(a+1)\sigma_{1}k^{a}+(n-1)k_{1}+k_{2}+\gamma_{1}\lambda_{1}
\Big)^{2}\\&&
+\Big(\frac{c_{0}}{R}(n-1)(\sqrt{k_{1}}+\frac{2}{R})+\frac{3c_{1}}{R^{2}}+c_{2}k_{2} \Big)^{2},
\end{eqnarray*}
  at point $(x_{3},t_{3})$. As $\eta(x,t)=1$ in $Q_{2R, T_{1}}$  and using inequality $\sqrt{x+y}\leq\sqrt{x}+\sqrt{y}$ for $x,y\geq0$, we have
\begin{eqnarray*}
\frac{|\nabla u|}{u(\rho-\ln u)}&\leq&
6^{\frac{1}{4}}\big(\ln \frac{k}{\tilde{k}}\big)\frac{\sqrt{c_{1}}}{R}
+(\frac{3}{4})^{\frac{1}{4}}\Big(2\lambda_{1} \gamma_{2}+2k^{a}\sigma_{2}\Big)^{\frac{1}{3}}\\&&
+\sqrt{2}\Big(\gamma_{1}\lambda_{2}+(a+1)\sigma_{1}k^{a}+(n-1)k_{1}+k_{2}+\gamma_{1}\lambda_{1}
\Big)^{\frac{1}{2}}\\&&
+\Big(\frac{c_{0}}{R}(n-1)(\sqrt{k_{1}}+\frac{2}{R})+\frac{3c_{1}}{R^{2}}+c_{2}k_{2} \Big)^{\frac{1}{2}}.
\end{eqnarray*}
The above inequality and $\rho-\ln u\leq 1+\ln \frac{k}{\tilde{k}}$ imply that (\ref{et4}).
\end{proof}
Similar to Corollary \ref{c1},  we can deduce  a global gradient estimate  from Theorem \ref{t4} by taking $R\to +\infty$ as follows.
\begin{corollary}\label{c5}
Let $(M, g(0),e^{-\phi_{0}}dv)$  be a complete noncompact weighted Riemannian manifold without boundary, and let $g(t), \phi(t)$ evolve by (\ref{e2}) for $t\in [0,T]$. Let $u$ be a positive solution to (\ref{e1}) in $M$ such that $\tilde{k}^{3}\leq u\leq k^{3}$ for some positive constants $k$ and $\tilde{k}$. Suppose that there exist constants $k_{1}, k_{2}$ such that
\begin{equation*}
Ric_{\phi}\geq -(n-1)k_{1}g,\qquad h\geq -k_{2}g.
\end{equation*}
on $M$. Then   there exist positive constants $c_{0},c_{1},$ and $c_{2}$ such that
\begin{eqnarray*}
\frac{|\nabla u|}{u}&\leq&(1+\ln \frac{k}{\tilde{k}})\Big[
(\frac{3}{4})^{\frac{1}{4}}\Big(2\Lambda_{1} \Gamma_{2}+2k^{a}\Sigma_{2}\Big)^{\frac{1}{3}}+\sqrt{c_{2}k_{2}}\\\nonumber&&
+\sqrt{2}\Big(\Gamma_{1}\Lambda_{2}+(a+1)\Sigma_{1}k^{a}+(n-1)k_{1}+k_{2}+\Gamma_{1}\Lambda_{1}
\Big)^{\frac{1}{2}}
\Big].
\end{eqnarray*}
\end{corollary}
\section{The fourth gradient estimate}
In this section  as  in  \cite{YLZ, FYZ} we obtain another type gradient estimate for positive solution of (\ref{e1}) under the geometric flow (\ref{e2}).

\begin{theorem}\label{t5}
Let $(M,g(0),e^{-\phi_{0}}dv )$ be a complete weighted Riemannian manifold, and let $g(t), \phi(t)$ evolve by (\ref{e2}) for $t\in [0,T]$. Given $x_{0}$ and $R>0$, let $u$ be a positive solution to (\ref{e1}) in $Q_{2R,T}$ such that $\tilde{k}\leq u\leq k$ for some positive constants $k$ and $\tilde{k}$. Suppose that there exist constants $k_{1}, k_{2}$ such that
\begin{equation*}
Ric_{\phi}\geq -(m-1)k_{1}g,\qquad -k_{2}g\leq h\leq k_{3}g,\qquad |\nabla h|\leq k_{4},
\end{equation*}
on $Q_{2R,T}$. Then  there exist positive constants $c_{0},c_{1},$ and $c_{2}$ such that
\begin{eqnarray*}
\frac{|\nabla u|^{2}}{u^{2}}-\frac{1}{b}\frac{u_{t}}{u}-qu^{a}-p\hat{A}\leq \frac{4m}{t}+K_{4}
\end{eqnarray*}
where
\begin{eqnarray*}\nonumber
K_{4}&=& 4m\Big[\frac{2mc_{1}}{bR^{2}}+
a\sigma_{1}    k^{-\frac{a}{b}}+b  \sigma_{1}k\lambda_{2} +\Big(\frac{c_{0}}{R}(n-1)(\sqrt{k_{1}}+\frac{2}{R})+\frac{3c_{1}}{R^{2}}+c_{2}k_{2} \Big)\Big]\\&&+\frac{2m}{b}\sqrt{
\frac{20}{5+8m}}\Big(  2(m-1)k_{1}+
a(a+2){\tilde k}^{-\frac{a}{b}}\sigma_{1} +b^{2} \gamma_{1}k\lambda_{2}+b^{2}\gamma_{1}k^{2}\lambda_{3}\\&& +2(1-b)k_{2} \Big)+\frac{\sqrt{3m}}{b}
\big(\frac{20m}{5+8m}\big)^{\frac{1}{6}}(2ab\sigma_{2})^{\frac{2}{3}}{\tilde k}^{-\frac{2a}{3b}}
+2\sqrt{m\lambda_{1}\gamma_{3}}\\&&+\frac{\sqrt{3m}}{b}
\big(\frac{20m}{5+8m}\big)^{\frac{1}{6}}(2b^{2}\lambda_{2}\gamma_{2})^{\frac{2}{3}}
+2n\sqrt{2m}(k_{2}+k_{3})\\&&+\frac{\sqrt{3m}}{b}
\big(\frac{20m}{5+8m}\big)^{\frac{1}{6}}\Big(3b\sqrt{n}k_{4}+2bk_{3}\theta_{1}+b\theta_{2}\Big)^{\frac{2}{3}},
\end{eqnarray*}
for $0<b\leq 1$ and if $b>1$ then it is enough in the above equality to replace $(1-b)k_{2}$ with $(b-1)k_{3}$.
\end{theorem}
Before we prove Theorem \ref{t5}, firstly we derive the following lemma.
\begin{lemma}\label{l5}
Let  $(M^{n},g(t),e^{-\phi}dv)$  be a weighted  Riemannian manifold, $g(t)$ evolves by (\ref{e2}) for $t\in[0,T]$ satisfies the hypotheses of Theorem \ref{t5}. If $s=u^{-b}$  and  $S=\frac{|\nabla s|^{2}}{s^{2}}+b\frac{s_{t}}{s}-b^{2}qs^{-\frac{a}{b}}-b^{2}p\hat{A}$ where $b$ is a given positive constant, then  we have
\begin{eqnarray}\nonumber
(\Delta_{\phi}-\partial_{t})S&\geq&\frac{1}{4m}\big(\frac{S}{b}+|\nabla \ln s|^{2}\big) ^{2}+\frac{2}{5}|\nabla \ln s|^{4}+\frac{2}{b}\langle \nabla S, \nabla \ln s\rangle\\\nonumber&&
-2(m-1)k_{1}|\nabla \ln s|^{2}-a(a+2)s^{-(2+\frac{a}{b})}q|\nabla s|^{2}\\\nonumber&&+ 2abs^{-(1+\frac{a}{b})}\langle \nabla q, \nabla s\rangle+aqs^{-\frac{a}{b}}S-b^{2}\hat{A}\Delta_{\phi}p-b^{2}ps\hat{A}_{s}\big(\frac{S}{b}+|\nabla \ln s|^{2}\big)\\\label{el5}&&
-b^{2}p\hat{A}_{ss}|\nabla s|^{2}-2b^{2}\hat{A}_{s}\langle \nabla s,\nabla p\rangle\\\nonumber&&
-2(1-b)k_{2}|\nabla \ln s|^{2}-2b^{2}n(k_{2}+k_{3})^{2}\\\nonumber&&-\Big(3b\sqrt{n}k_{4}+2bk_{3}\theta_{1}+b\theta_{2}\Big)|\nabla \ln s|,
\end{eqnarray}
for $0<b\leq 1$ and if $b>1$ then it is enough in the above equality to replace $(1-b)k_{2}$ with $(b-1)k_{3}$.
\end{lemma}
\begin{proof}
Since  $s=u^{-b}$ we have $s_{t}=-bu^{-b-1}u_{t}$, $\nabla s=-bu^{-b-1}\nabla u$ and
\begin{equation*}
\Delta_{\phi}s=-bu^{-b-1}\Delta_{\phi}u+b(b+1)u^{-b-2}|\nabla u|^{2}.
\end{equation*}
Thus, from (\ref{e3}) we get
\begin{equation}\label{e1s5}
(\Delta_{\phi}-\partial_{t})s=-bps \hat{A}-bqs^{1-\frac{a}{b}}+\frac{b+1}{b}\frac{ |\nabla s|^{2}}{s}.
\end{equation}
By the weighted Bochner formula  and Lemma \ref{l1} we have
\begin{eqnarray*}
\Delta_{\phi}S&=&2s^{-2}|{\rm Hess} s|^{2}+2s^{-2}\langle \nabla \Delta_{\phi}s, \nabla s\rangle+2s^{-2}Ric_{\phi}(\nabla s,\nabla s)\\&&-8s^{-3}{\rm Hess} s(\nabla s, \nabla s)-2s^{-3}|\nabla s|^{2}\Delta_{\phi}s+6s^{-4}|\nabla s|^{4}-b^{2}s^{-\frac{a}{b}}\Delta_{\phi}q\\&&
+2abs^{-(1+\frac{a}{b})}\langle \nabla q, \nabla s\rangle-a(a+b)s^{-(2+\frac{a}{b})}q|\nabla s|^{2}+abs^{-(1+\frac{a}{b})}q\Delta_{\phi}s\\&&-b^{2}\Delta_{\phi}(p\hat{A})
+bs^{-1}(\Delta_{\phi} s)_{t}-bs^{-2}(|\nabla s|^{2})_{t}-2bs^{-2}h(\nabla s,\nabla s)\\&&-bs^{-2}s_{t}\Delta_{\phi} s+2bs^{-3}s_{t}|\nabla s|^{2}
+2bs^{-1}\langle h,{\rm Hess} s\rangle\\&&+2bs^{-1}\langle {\rm div}h-\frac{1}{2}\nabla({\rm tr}_{g}h),\nabla s\rangle -2bs^{-1}h(\nabla \phi,\nabla s)+bs^{-1}\langle \nabla s, \nabla \Delta \phi\rangle.
\end{eqnarray*}
Using again Lemma \ref{l1} we obtain
\begin{eqnarray*}
\partial_{t}S&=&-2s^{-2}h(\nabla s,\nabla s)+2s^{-2}\langle \nabla s_{t}, \nabla s\rangle-2s^{-3}s_{t}|\nabla s|^{2}-b^{2}q_{t}s^{-\frac{a}{b}}\\&&+abqs^{-(1+\frac{a}{b})}s_{t}-b^{2}(p\hat{A})_{t}+bs^{-1}s_{tt}-bs^{-2}s_{t}^{2}.
\end{eqnarray*}
Combining the above two equalities, we derive
\begin{eqnarray*}
(\Delta_{\phi}-\partial_{t})S&=&2s^{-2}|{\rm Hess} s|^{2}+2s^{-2}\langle \nabla (\Delta_{\phi}-\partial_{t})s, \nabla s\rangle+2s^{-2}Ric_{\phi}(\nabla s,\nabla s)\\&&-8s^{-3}{\rm Hess} s(\nabla s, \nabla s)-2s^{-3}|\nabla s|^{2}(\Delta_{\phi}-\partial_{t})s+6s^{-4}|\nabla s|^{4}\\&&-b^{2}s^{-\frac{a}{b}}(\Delta_{\phi}-\partial_{t})q
+2abs^{-(1+\frac{a}{b})}\langle \nabla q, \nabla s\rangle\\&&-a(a+b)s^{-(2+\frac{a}{b})}q|\nabla s|^{2}+abs^{-(1+\frac{a}{b})}q(\Delta_{\phi}-\partial_{t})s\\&&-b^{2}(\Delta_{\phi}-\partial_{t})(p\hat{A})+2(1-b)s^{-2}h(\nabla s,\nabla s)
+bs^{-1}\partial_{t} (\Delta_{\phi}- \partial_{t} )s\\&&-bs^{-2}|\nabla s|^{2}-bs^{-2}s_{t}(\Delta_{\phi} -\partial_{t} )s+2bs^{-3}s_{t}|\nabla s|^{2}\\&&
+2bs^{-1}\langle h,{\rm Hess} s\rangle+2bs^{-1}\langle {\rm div}h-\frac{1}{2}\nabla({\rm tr}_{g}h),\nabla s\rangle\\&& -2bs^{-1}h(\nabla \phi,\nabla s)+bs^{-1}\langle \nabla s, \nabla \Delta \phi\rangle.
\end{eqnarray*}
Substituting (\ref{e1s5}) in above identity we  infer
\begin{eqnarray}\nonumber
(\Delta_{\phi}-\partial_{t})S&=&2s^{-2}\Big(|{\rm Hess} s|^{2}-2s^{-1}{\rm Hess} s(\nabla s,\nabla s)+s^{-2}|\nabla s|^{4} \Big)+\frac{2}{bs}\langle \nabla S, \nabla s\rangle\\\nonumber&&
+2s^{-2}Ric_{\phi}(\nabla s,\nabla s)-a(a+2)s^{-(2+\frac{a}{b})}q|\nabla s|^{2}\\\nonumber&&+ 2abs^{-(1+\frac{a}{b})}\langle \nabla q, \nabla s\rangle+aqs^{-\frac{a}{b}}S-b^{2}\hat{A}\Delta_{\phi}p-b^{2}p\hat{A}_{s}\Delta_{\phi}s\\\label{e2s5}&&-b^{2}p\hat{A}_{ss}|\nabla s|^{2}-2b^{2}\hat{A}_{s}\langle \nabla s,\nabla p\rangle+2(1-b)s^{-2}h(\nabla s,\nabla s)\\\nonumber&&
+2bs^{-1}\langle h,{\rm Hess} s\rangle+2bs^{-1}\langle {\rm div}h-\frac{1}{2}\nabla({\rm tr}_{g}h),\nabla s\rangle\\\nonumber&& -2bv^{-1}h(\nabla \phi,\nabla s)+bs^{-1}\langle \nabla s, \nabla \Delta \phi\rangle.
\end{eqnarray}
 If $b\leq 1$ then   from (\ref{07})-(\ref{9}), for  $\epsilon=\frac{1}{4b}$ we have
\begin{eqnarray}\nonumber
&&2(1-b)s^{-2}h(\nabla s,\nabla s)
+2bs^{-1}\langle h,{\rm Hess} s\rangle+2bs^{-1}\langle {\rm div}h-\frac{1}{2}\nabla({\rm tr}_{g}h),\nabla s\rangle\\\nonumber&& -2bs^{-1}h(\nabla \phi,\nabla s)+bs^{-1}\langle \nabla s, \nabla \Delta \phi\rangle\\\label{e3s5}&&\geq
-2(1-b)k_{2}|\nabla \ln s|^{2}-\frac{1}{2} s^{-2}|{\rm Hess} s|^{2}-2b^{2}n(k_{2}+k_{3})^{2}\\\nonumber&&-3b\sqrt{n}k_{4}|\nabla \ln s|-2bk_{3}\theta_{1}|\nabla \ln s|-b\theta_{2}|\nabla \ln s|
\end{eqnarray}
By Cauchy's inequality,  we have
\begin{equation}\label{e4s5}
2s^{-3}{\rm Hess }s(\nabla s,\nabla s)\leq \frac{5}{4} s^{-2}|{\rm Hess}s|^{2}+\frac{4}{5}|\nabla \ln s|^{4}.
\end{equation}
Plugging (\ref{e3s5}) and (\ref{e3s5}) into (\ref{e2s5}) we obtain
\begin{eqnarray}\nonumber
(\Delta_{\phi}-\partial_{t})S&\geq&\frac{1}{4}s^{-2}|{\rm Hess} s|^{2}+\frac{2}{5}|\nabla \ln s|^{4}+\frac{2}{b}\langle \nabla S, \nabla \ln s\rangle\\\nonumber&&
+2s^{-2}Ric_{\phi}(\nabla s,\nabla s)-a(a+2)s^{-(2+\frac{a}{b})}q|\nabla s|^{2}\\\nonumber&&+ 2abs^{-(1+\frac{a}{b})}\langle \nabla q, \nabla s\rangle+aqs^{-\frac{a}{b}}S-b^{2}\hat{A}\Delta_{\phi}p-b^{2}p\hat{A}_{s}\Delta_{\phi}s\\\label{e5s5}&&
-b^{2}p\hat{A}_{ss}|\nabla s|^{2}-2b^{2}\hat{A}_{s}\langle \nabla s,\nabla p\rangle\\\nonumber&&
-2(1-b)k_{2}|\nabla \ln s|^{2}-2b^{2}n(k_{2}+k_{3})^{2}\\\nonumber&&-3b\sqrt{n}k_{4}|\nabla \ln s|-2bk_{3}\theta_{1}|\nabla \ln s|-b\theta_{2}|\nabla \ln s|.
\end{eqnarray}
Applying (\ref{10})  and  $s^{-1}\Delta_{\phi} s=\frac{S}{b}+|\nabla \ln s|^{2}$ to (\ref{e5s5}), we arrive at (\ref{el5}). If $b>1$ then it is enough in the above equality to replace $(1-b)k_{2}$ with $(b-1)k_{3}$.
\end{proof}
\begin{proof}[Proof of theorem \ref{t5}]
We Choose $\psi$ and $\eta$  as in the proof of Theorem \ref{t1} and define   $W(x,t)=tS(x,t)$. For any $T_{1}\in(0,T]$, let $(x_{4},t_{4})\in Q_{2R,T_{1}}$ be a point  where $\mathcal{B}(x,t)=\eta(x,t)W(x,t)$ achieve  its maximum, and without loss of generality we can assume $b\leq1$ and  $\mathcal{B}(x_{4},t_{4})>0$, and then  $W(x_{4},t_{4})>0$. By  a similar argument as in the proof of Theorem \ref{t1},  at   point $(x_{4},t_{4})$  we have
\begin{eqnarray*}
0&\geq&  (\Delta_{\phi}-\partial_{t})\mathcal{B}= (\Delta_{\phi}-\partial_{t})(\eta W)\\&=&
W(\Delta_{\phi}-\partial_{t})\eta+\eta(\Delta_{\phi}-\partial_{t})W+2\langle\nabla\eta,\nabla W\rangle\\&\geq&
\eta t (\Delta_{\phi}-\partial_{t})S-\eta S
 -\Big(\frac{c_{0}}{R}(n-1)(\sqrt{k_{1}}+\frac{2}{R})+\frac{3c_{1}}{R^{2}}+c_{2}k_{2} \Big)W.
\end{eqnarray*}
Multiplying the  inequality by $ t\eta$  on both sides  and using Lemma \ref{l5} we can write
\begin{eqnarray}\nonumber
0&\geq&
\frac{1}{4m}\eta^{2}\big(\frac{W}{b}+t|\nabla \ln s|^{2}\big) ^{2}+\frac{2}{5}(\eta t)^{2}|\nabla \ln s|^{4}+\frac{2(\eta t)^{2}}{b}\langle \nabla S, \nabla \ln s\rangle\\\nonumber&&
-2(m-1)k_{1}(\eta t)^{2}|\nabla \ln s|^{2}-a(a+2)(\eta t)^{2} s^{-\frac{a}{b}}q|\nabla \ln s|^{2}\\\nonumber&&+ 2ab(\eta t)^{2}s^{-\frac{a}{b}}\langle \nabla q, \nabla \ln s\rangle+aq(\eta t)^{2}s^{-\frac{a}{b}}S-(\eta t)^{2}b^{2}\hat{A}\Delta_{\phi}p\\\nonumber&&
-b^{2}\eta^{2} t ps\hat{A}_{s}\big(\frac{W}{b}+t|\nabla \ln s|^{2}\big)
-b^{2}(\eta t)^{2}ps^{2}\hat{A}_{ss}|\nabla \ln s|^{2}\\\nonumber&&-2b^{2}(\eta t)^{2}s\hat{A}_{s}\langle \nabla  \ln s,\nabla p\rangle
-2(1-b)k_{2}(\eta t)^{2}|\nabla \ln s|^{2}-2nb^{2}(\eta t)^{2}(k_{2}+k_{3})^{2}\\\label{1et5}&&-(\eta t)^{2}\Big(3b\sqrt{n}k_{4}+2bk_{3}\theta_{1}+b\theta_{2}\Big)|\nabla \ln s|-t\eta^{2} S\\\nonumber&&
 -t\eta\Big(\frac{c_{0}}{R}(n-1)(\sqrt{k_{1}}+\frac{2}{R})+\frac{3c_{1}}{R^{2}}+c_{2}k_{2} \Big)W,
\end{eqnarray}
at   point $(x_{4},t_{4})$ .
Now we obtain  a lower bound for each term of the right-hand side of above inequality. In point $(x_{4},t_{4})$, Cauchy's inequality and Young's inequality imply that
\begin{eqnarray}\nonumber
\frac{2(\eta t)^{2}}{b}\langle \nabla S, \nabla \ln s\rangle&=&-\frac{2\eta t}{b}W\langle \nabla \eta, \nabla \ln s\rangle\geq- \frac{2}{b}\eta^{\frac{3}{2}}t \frac{\sqrt{c_{1}}}{R}W| \nabla \ln s|\\\label{2et5}&\geq&-\frac{t}{2mb}\eta \mathcal{B}|\nabla\ln s|^{2}-\frac{2mtc_{1}}{bR^{2}}\mathcal{B},
\end{eqnarray}
and
\begin{eqnarray}\nonumber
&&\Big(  2(m-1)k_{1}+ a(a+2) s^{-\frac{a}{b}}q+b^{2} ps\hat{A}_{s}+b^{2}ps^{2}\hat{A}_{ss}+2(1-b)k_{2}\Big)|\nabla \ln s|^{2}\\\label{3et5}&&\leq
\Big(  2(m-1)k_{1}+ a(a+2){\tilde k}^{-\frac{a}{b}}\sigma_{1}+b^{2} \gamma_{1}k\lambda_{2}+b^{2}\gamma_{1}k^{2}\lambda_{3}
+2(1-b)k_{2} \Big)|\nabla \ln s|^{2}
\\\nonumber&&\leq\frac{5+8m}{80m}|\nabla \ln s|^{4}+\frac{20m}{5+8m}\Big(  2(m-1)k_{1}+
a(a+2){\tilde k}^{-\frac{a}{b}}\sigma_{1} +b^{2} \gamma_{1}k\lambda_{2}\\\nonumber&&
+b^{2}\gamma_{1}k^{2}\lambda_{3} +2(1-b)k_{2} \Big)^{2}.
\end{eqnarray}
Using Young's inequality again we arrive at
\begin{eqnarray}\nonumber
 2abs^{-\frac{a}{b}}\langle \nabla q, \nabla \ln s\rangle&\leq&  2ab{\tilde k}^{-\frac{a}{b}}\sigma_{2} |\nabla \ln s|\\\label{4et5}&\leq&\frac{5+8m}{80m}|\nabla \ln s|^{4}+\frac{3}{4}\big(\frac{20m}{5+8m}\big)^{\frac{1}{3}}(2ab\sigma_{2})^{\frac{4}{3}}{\tilde k}^{-\frac{4a}{3b}}.
\end{eqnarray}
Also, we have
$
aq(\eta t)^{2}s^{-\frac{a}{b}}S\geq -a\sigma_{1}  \eta t k^{-\frac{a}{b}}\mathcal{B}
$,
$
b^{2}\hat{A}\Delta_{\phi}p\leq  b^{2}\lambda_{1}\gamma_{3}
$, and
\begin{eqnarray}\nonumber
b\eta^{2} t ps\hat{A}_{s}W\leq b t\eta \sigma_{1}k\lambda_{5}\mathcal{B}.
\end{eqnarray}
According Young's inequality,
\begin{eqnarray}\nonumber
2b^{2}s\hat{A}_{s}\langle \nabla \ln s,\nabla p\rangle&\leq& 2b^{2}\lambda_{2}\gamma_{2}|\nabla \ln s|
\\\label{5et5}&\leq&
\frac{5+8m}{80m}|\nabla \ln s|^{4}+\frac{3}{4}\big(\frac{20m}{5+8m}\big)^{\frac{1}{3}}(2b^{2}\lambda_{2}\gamma_{2})^{\frac{4}{3}}
\end{eqnarray}
and
\begin{eqnarray}\nonumber
&&\Big(3b\sqrt{n}k_{4}+2bk_{3}\theta_{1}+b\theta_{2}\Big)|\nabla \ln s|\\\label{6et5}&&\leq \frac{5+8m}{80m}|\nabla \ln s|^{4}+\frac{3}{4}\big(\frac{20m}{5+8m}\big)^{\frac{1}{3}}\Big(3b\sqrt{n}k_{4}+2bk_{3}\theta_{1}+b\theta_{2}\Big)^{\frac{4}{3}}.
\end{eqnarray}
Substituting (\ref{2et5})-(\ref{6et5})  into (\ref{1et5}) yields
\begin{eqnarray*}\nonumber
0&\geq&\frac{1}{4mb^{2}}\mathcal{B}^{2}-\Big[ \frac{2mtc_{1}}{bR^{2}}+
a\sigma_{1}  \eta t k^{-\frac{a}{b}}+b t\eta \sigma_{1}k\lambda_{2}+\eta \\\nonumber&&+t\Big(\frac{c_{0}}{R}(n-1)(\sqrt{k_{1}}+\frac{2}{R})+\frac{3c_{1}}{R^{2}}+c_{2}k_{2} \Big)
\Big]\mathcal{B}
\\&&-(\eta t)^{2}\frac{20m}{5+8m}\Big(  2(m-1)k_{1}+
a(a+2){\tilde k}^{-\frac{a}{b}}\sigma_{1} +b^{2} \gamma_{1}k\lambda_{2}\\\nonumber&&+b^{2}\gamma_{1}k^{2}\lambda_{3} +2(1-b)k_{2} \Big)^{2}\\&&-
(\eta t)^{2}\frac{3}{4}\big(\frac{20m}{5+8m}\big)^{\frac{1}{3}}(2ab\sigma_{2})^{\frac{4}{3}}{\tilde k}^{-\frac{4a}{3b}}
-(\eta t)^{2}b^{2}\lambda_{1}\gamma_{3}\\&&
-(\eta t)^{2}\frac{3}{4}\big(\frac{20m}{5+8m}\big)^{\frac{1}{3}}(2b^{2}\lambda_{2}\gamma_{2})^{\frac{4}{3}}
-2nb^{2}(\eta t)^{2}(k_{2}+k_{3})^{2}\\&&
-(\eta t)^{2}\frac{3}{4}\big(\frac{20m}{5+8m}\big)^{\frac{1}{3}}\Big(3b\sqrt{n}k_{4}+2bk_{3}\theta_{1}+b\theta_{2}\Big)^{\frac{4}{3}},
\end{eqnarray*}
at   point $(x_{4},t_{4})$.
Set
\begin{eqnarray*}\nonumber
C_{3}= \frac{2mc_{1}}{bR^{2}}+
a\sigma_{1}  \eta  k^{-\frac{a}{b}}+b \eta \sigma_{1}k\lambda_{2} +\Big(\frac{c_{0}}{R}(n-1)(\sqrt{k_{1}}+\frac{2}{R})+\frac{3c_{1}}{R^{2}}+c_{2}k_{2} \Big)
\end{eqnarray*}
and
\begin{eqnarray*}\nonumber
E_{3}&=&\frac{20m}{5+8m}\Big(  2(m-1)k_{1}+
a(a+2){\tilde k}^{-\frac{a}{b}}\sigma_{1} +b^{2} \gamma_{1}k\lambda_{2}+b^{2}\gamma_{1}k^{2}\lambda_{3} +2(1-b)k_{2} \Big)^{2}\\&&+
\frac{3}{4}\big(\frac{20m}{5+8m}\big)^{\frac{1}{3}}(2ab\sigma_{2})^{\frac{4}{3}}{\tilde k}^{-\frac{4a}{3b}}
+b^{2}\lambda_{1}\gamma_{3}\\&&
+\frac{3}{4}\big(\frac{20m}{5+8m}\big)^{\frac{1}{3}}(2b^{2}\lambda_{2}\gamma_{2})^{\frac{4}{3}}
+2nb^{2}(k_{2}+k_{3})^{2}\\&&
+\frac{3}{4}\big(\frac{20m}{5+8m}\big)^{\frac{1}{3}}\Big(3b\sqrt{n}k_{4}+2bk_{3}\theta_{1}+b\theta_{2}\Big)^{\frac{4}{3}}.
\end{eqnarray*}
Hence,
\begin{equation*}
0\geq\frac{1}{4mb^{2}}\mathcal{B}^{2}-(\eta +tC_{3})\mathcal{B}-(\eta t)^{2}E_{3},
\end{equation*}
at   point $(x_{4},t_{4})$.
For a positive number $\tilde{a}$ and two nonnegative  numbers $\tilde{b},\tilde{c}$,  the  quadratic inequality of the form $\tilde{a}x^{2}-\tilde{b}x-\tilde{c}\leq0$ implies that $x\leq\frac{\tilde{b}}{\tilde{a}}+\sqrt{\frac{\tilde{c}}{\tilde{a}}}$, Therefore,
\begin{equation*}
\mathcal{B}\leq4mb^{2}(\eta +tC_{3})+2b\eta t\sqrt{mE_{3}},
\end{equation*}
at   point $(x_{4},t_{4})$.
Since $\mathcal{B}=\eta t S$ and  $\eta(x,T_{1})=1$, we infer
\begin{equation*}
S\leq\frac{4mb^{2}}{t}+C_{4}+2b\sqrt{mE_{3}}
\end{equation*}
at   point $(x_{4},t_{4})$ where
\begin{eqnarray*}\nonumber
C_{4}= 4mb^{2}\Big[\frac{2mc_{1}}{bR^{2}}+
a\sigma_{1}    k^{-\frac{a}{b}}+b  \sigma_{1}k\lambda_{2} +\Big(\frac{c_{0}}{R}(n-1)(\sqrt{k_{1}}+\frac{2}{R})+\frac{3c_{1}}{R^{2}}+c_{2}k_{2} \Big)\Big].
\end{eqnarray*}
Since $T_{1}$ is arbitrary, and using inequality $\sqrt{x+y}\leq\sqrt{x}+\sqrt{y}$ for $x,y\geq0$, we complete the proof.

\end{proof}
Similar to Corollary \ref{c1},  we can deduce  a global gradient estimate  from Theorem \ref{t5} by taking $R\to +\infty$ as follows.
\begin{corollary}\label{c6}
Let $(M, g(0),e^{-\phi_{0}}dv)$  be a complete noncompact weighted Riemannian manifold without boundary, and let $g(t), \phi(t)$ evolve by (\ref{e2}) for $t\in [0,T]$. Let $u$ be a positive solution to (\ref{e1}) in $M$ such that $\tilde{k}\leq u\leq k$ for some positive constants $k$ and $\tilde{k}$. Suppose that there exist constants $k_{1}, k_{2}$ such that
\begin{equation*}
Ric_{\phi}\geq -(n-1)k_{1}g,\qquad -k_{2}g\leq h\leq k_{3}g,\qquad |\nabla h|\leq k_{4}
\end{equation*}
on $M$. Then  there exist positive constants $c_{0},c_{1},$ and $c_{2}$ such that
\begin{eqnarray*}
\frac{|\nabla u|^{2}}{u^{2}}-\frac{1}{b}\frac{u_{t}}{u}-qu^{a}-p\hat{A}\leq \frac{4m}{t}+K_{5}
\end{eqnarray*}
where
\begin{eqnarray*}\nonumber
K_{5}&=& 4m\Big[
a\sigma_{1}    k^{-\frac{a}{b}}+b  \sigma_{1}k\Lambda_{2} +c_{2}k_{2} \Big]+\frac{2m}{b}\sqrt{
\frac{20}{5+8m}}\Big(  2(m-1)k_{1}\\&&+
a(a+2){\tilde k}^{-\frac{a}{b}}\Sigma_{1} +b^{2} \Gamma_{1}k\Lambda_{2}+b^{2}\Gamma_{1}k^{2}\Lambda_{3} +2(1-b)k_{2} \Big)\\&&+\frac{\sqrt{3m}}{b}
\big(\frac{20m}{5+8m}\big)^{\frac{1}{6}}(2ab\Sigma_{2})^{\frac{2}{3}}{\tilde k}^{-\frac{2a}{3b}}
+2\sqrt{m\Lambda_{1}\Gamma_{3}}\\&&+\frac{\sqrt{3m}}{b}
\big(\frac{20m}{5+8m}\big)^{\frac{1}{6}}(2b^{2}\Lambda_{5}\Gamma_{2})^{\frac{2}{3}}
+2n\sqrt{2m}(k_{2}+k_{3})\\&&+\frac{\sqrt{3m}}{b}
\big(\frac{20m}{5+8m}\big)^{\frac{1}{6}}\Big(3b\sqrt{n}k_{4}+2bk_{3}\Theta_{1}+b\Theta_{2}\Big)^{\frac{2}{3}}.
\end{eqnarray*}
\end{corollary}
Using  the same arguments at in the proof of
 Corollary \ref{c2},   we can obtain the following Harnack inequality.
\begin{corollary}\label{c7}
With the same assumptions in Corollary \ref{c5},  for $(y_{1},s_{1})\in M\times (0,T]$ and $(y_{2},s_{2})\in M\times (0,T]$ such that $s_{1}<s_{2}$,   we have
\begin{equation*}
u(y_{1},s_{1})\leq u(y_{2},s_{2})\big( \frac{s_{2}}{s_{1}}\big)^{ 4mb}\exp\left\{ \frac{\mathcal{J}(y_{1},s_{1},y_{2},s_{2})}{4b}+(s_{2}-s_{1})b\big(  k \Sigma_{1}+ \Gamma_{1}\Lambda_{1}+ K_{5}\big)\right\}.
\end{equation*}
\end{corollary}


\begin{thebibliography}{99}
\bibitem{MBA} M. Bailesteanu, X. Cao, and A. Pulemotov, Gradient estimates for the heat equation under the Ricci flow, J. Funct. Anal. 258 (10)(2010), 3517-3542.

\bibitem{DB} D. Bakry and M. {\'E}mery, Diffusions hypercontractives, in: Seminaire de probablities XIX, 1983/84, in Lecture notes in Math. vol. 1123. Springer, Berlin, 1985, 177-206.

\bibitem{DB1}  D. Bakry and M. Ledoux, Sobolev inequalities and Myers' diameter theorem for an abstract Markov generator, Duke Math. J. 85(1)(1996), 253-270.


\bibitem{CA} E. Calabi, An extension of E. Hopf's maximum principle with an application to Riemannian geometry, Duke Math. J. 25(1)(1958),45-56.

\bibitem{XCH} X. Cao and R. S. Hamilton, Differential Harnack estimates for the time-dependent heat equations  with potentials, Geom. Funct. Anal., 19 (4) (2009), 989-1000.

\bibitem{HDC} H. D. Cao and X. P. Zhu, A complete proof  of the Poincar{\'e}  and geometrization conjecture-application of the Hamilton-Perelman theory of the Ricci flow, Asian J. Math., 10(2)(2006), 165-492.

\bibitem{GCL} G. Catino,  L. Cremaschi, Z. Djadli, C. Montegazza,  and L. Mazzier, The Ricci-Bourguignon flow, Pac. J. Math., 287(2) (2017),337-370.

\bibitem{LC}L. Chen and W. Chen, Gradient estimates for  a nonlinear parabolic equation on complete non-compact Riemannian manifolds, Ann. Glob. Anal. Geom. 35(4)(2009), 397-404.

\bibitem{QCH} Q. Chen and H. Qiu, Gradient estimates and Harnack inequalities  of a nonlinear parabolic equation for the $V$-Laplacian, Ann. global Anal. Geom., 50 (1)(2016), 47-64.

\bibitem{QCH} Q. Chen and G. Zhao, Li-Yau type and Souplet-Zhang type gradient  estimates of a parabolic equation for the $V$-Laplacian, J. Math. Anal. Appl., 463 (2) (2018), 744-759.

\bibitem{BC} B. Chow, The yamabe flow on locally conformally flat manifolds with positive Ricci curvature, Comm. Pure. Appl. Math., 45( 8)(1992), 1003-1014.

\bibitem{BC1} B. Chow, P. Lu, and L. Ni, Hamilton's Ricci flow, Graduate studies in mathematics, vol. 77, American mathematical society, 2006.

\bibitem{NTD} N. T. Dung, N. N. Khanh, Gradient estimates of Hamilton-Souplet-Zhang type for a general heat equation on Riemannian manifolds, Arch. Math. 105(5)(2015), 479-490.

\bibitem{HGU} H. Guo and M. Ishida, Harnack estimates for nonlinear backward heat equations in geometric flows, J. Funct. Anal., 267 (8) (2014), 2638-2662.

\bibitem{RH} R. S. Hamilton, Three-manifolds with positive Ricci flow, J. Differential Geom., 1(1982), 255-306.

\bibitem{RSH1}R. S. Hamilton, The harnack estimate for the Ricci flow, J. Differ. Geom., 37(1993), 225-243.

\bibitem{RSH2} R. S. Hamilton, A matrix Harnack estimate for the heat eqaution, Comm. Anal. Geom., 1(1)(1993), 113-126.

\bibitem{MIS} M. Ishida, Geometric flows and  differential Harnack estimates for heat equations with potentials, Ann, Global Anal. Geom., 45 (4) (2014), 287-302.

\bibitem{XDL}X. D. Li, Liouville theorems for symmetric diffusion operators on complete Riemannian manifolds, J. Math. Pures Appl. 84(10)(2005), 1295-1361.

\bibitem{YLI} Y. Li, Li-Yau-Hamilton estimates and Bakry-Emery-Ricci curvature, Nonlinear Anal., 113(2015), 1-32.

\bibitem{LY} P. Li and S. T. Yau, On the parabolic kernel of the Schr{\"o}dinger operator, Acta Math. 156 (1) (1986), 153-201.

\bibitem{YZ}Y. Li and X. Zhu, Harnack estimates for a nonlinear parabolic equation under Ricci flow, Differ. Geom. Appl. 56(2018), 67-80.

\bibitem{YLZ} Y. Li and X. Zhu, Li-Yau Harnack estimates for a heat-type equation under the geometric flow, Potential Anal., 52 (2020), 469-496.

\bibitem{BLI} B. List, Evolution of an extended Ricci flow system, Commun. Anal. Geom., 16(2008), 1007-1048.

\bibitem{LMA} L. Ma, Gradient estimates for a simple elliptic equation on complete noncompact Riemannian manifolds, J. Funct. Anal., 241 (1)(2006), 374-382.

\bibitem{ERN} E. R. Negrin, Gradient estimates and  a Liouville type Schr\"{o}dinger operator, J. Funct. Anal., 127 (1) (1995), 198-203.

\bibitem{GP} G. Perelman, The entropy formula for the Ricci flow and its geometric applications, arXiv: math/021159 [math. DG], 2002.

  \bibitem{FR} F. Rothe, Global solutions of reaction-diffusion  system, Springer, 1984.

\bibitem{QR} Q. Ruan, Elliptic-type gradient estimate for Schr\"{o}dinger equations on noncompact manifolds, Bull. Lond. Math. Soc., 39 (6) (2007), 982-988.

\bibitem{S} R. Schoen and S.-T. Yau, Lecture on  differential geometry, International Press, Cambridge, MA, 1994.

\bibitem{JS}J. Smoller, Shock  waves and reaction-diffusion equations, Springer-Verlag, 1983.

\bibitem{SZ}  P. Souplet and Q. S. Zhang, Sharp gradient estimate and Yau's Liouville theorem for the heat equation  on noncompact manifolds, Bull. Lond. Math. Soc. 38(6)(2006), 1045-1053.

\bibitem{JS} J. Sun, Graqdient estimates for positive solutions of the heat equation under geometric flow, Pacific J. Math. 253(2) (2011), 489-510.


\bibitem{JYW} J. Y. Wu, Li-Yau type estimates for  a nonlinear parabolic equation  on  complete manifolds, J. Math. Anal. Appl., 369 (2010), 400-407.

\bibitem{JYWU} J.-Y. Wu, Elliptic gradient estimates for  a nonlinear heat equation and applications, Nonlinear Anal., 151 (2017), 1-17.

\bibitem{YY} Y. Yang, Gradient estimates for a nonlinear parabolic equation on Riemannian manifolds, Proc. Amer. Math. Soc. 136(11)(2008), 4095-4102.

\bibitem{FYZ} F. Yang and  L. Zhang, Gradient estimates and Harnack inequalities for a nonlinear parabolic equation on smooth metric measure spaces, J. Differential equations, 268 (8) (2020), 4577-4617.


\bibitem{GZH} G. Zhao, Gradient estimates and Harnack inequalities of a parabolic equation under geometric flow, J. Math. Anal. Appl., 483 (2020), 123631.

\bibitem{ZYL} X. Zhu and Y. Li, Li-Yau estimates for  a nonlinear parabolic equation on manifolds, Math. Phys. Anal. Geom., 17 (2014), 273-288.



\end{thebibliography}
\end{document}